\documentclass[a4paper, 11pt]{amsart}   
\usepackage{mathptmx, amssymb,amscd,latexsym, eulervm}
\usepackage[all]{xy}
\usepackage{changepage} 
\usepackage{amsmath}
\usepackage{amsthm}
\usepackage{mathdots}

\usepackage[onehalfspacing]{setspace}
\usepackage{tabularx}
\usepackage{amsfonts}
\usepackage{paralist}
\usepackage{aliascnt}
\usepackage[initials, lite]{amsrefs}
\usepackage{amscd}
\usepackage{blkarray}
\usepackage{mathbbol}
\usepackage{setspace}
\usepackage[inner=2.4cm,outer=2.4cm, bottom=3.2cm]{geometry}
\usepackage{tikz, tikz-cd}
\usepackage{calligra,mathrsfs}
\usepackage{comment}

\usepackage{tikz}
\usetikzlibrary{matrix}
\usetikzlibrary{arrows,calc}
\usetikzlibrary{decorations.markings,intersections,positioning,calc}

  \tikzset{mylabel/.style  args={at #1 #2  with #3}{
    postaction={decorate,
    decoration={
      markings,
      mark= at position #1
      with  \node [#2] {#3};
 } } } }
\allowdisplaybreaks

\BibSpec{collection.article}{%
	+{}  {\PrintAuthors}                {author}
	+{,} { \textit}                     {title}
	+{.} { }                            {part}
	+{:} { \textit}                     {subtitle}
	+{,} { \PrintContributions}         {contribution}
	+{,} { \PrintConference}            {conference}
	+{}  {\PrintBook}                   {book}
	+{,} { }                            {booktitle}
	+{,} { }                            {series}
	+{, vol.} { }                            {volume}
	+{,} { }                            {publisher}
	+{,} { \PrintDateB}                 {date}
	+{,} { pp.~}                        {pages}
	+{,} { }                            {status}
	+{,} { \PrintDOI}                   {doi}
	+{,} { available at \eprint}        {eprint}
	+{}  { \parenthesize}               {language}
	+{}  { \PrintTranslation}           {translation}
	+{;} { \PrintReprint}               {reprint}
	+{.} { }                            {note}
	+{.} {}                             {transition}
	+{}  {\SentenceSpace \PrintReviews} {review}
}
\AtBeginDocument{%
	\def\MR#1{}
}

\makeatletter
\@namedef{subjclassname@2020}{%
	\textup{2020} Mathematics Subject Classification}
\makeatother

\newcommand{\lk}{\mathrm{lk}}

\newcommand{\tr}{\operatorname{tr}}

\def\opn#1#2{\def#1{\operatorname{#2}}}
\opn\Cl{Cl} \opn\deg{deg} \opn\Stab{Stab} \opn\aff{aff} \opn\div{div}
\opn\cone{cone} \opn\End{End} \opn\mod{mod}  \opn\pdim{pdim} \opn\diag{diag} \opn\vert{vert} \opn\m{m} \opn\V{V}
\opn\Cone{Cone} \opn\Pyr{Pyr} \opn\max{max} \opn\min{min} \opn\int{int} \opn\rev{rev} \opn\ker{ker} \opn\lat{lat} \opn\pull{pull}
\opn\cok{coker} \opn\ant{ant}
\opn\inte{int}

\newcommand{\kk}{\mathbb{k}}

\newcommand{\NN}{\normalfont\mathbb{N}}
\newcommand{\ZZ}{\normalfont\mathbb{Z}}

\newcommand{\MM}{{\normalfont\mathfrak{m}}}

\newcommand{\mm}{{\normalfont\mathfrak{m}}}
\newcommand{\fkn}{{\normalfont\mathfrak{n}}}

\newcommand{\pp}{{\normalfont\mathfrak{p}}}

\newcommand{\depth}{\normalfont\text{depth}}

\newcommand{\grade}{\normalfont\text{grade}}

\newcommand{\Ext}{\normalfont\text{Ext}}

\newcommand{\Ker}{\normalfont\text{Ker}}

\newcommand{\ann}{\normalfont\text{Ann}}
\newcommand{\height}{{\normalfont\text{ht}}}
\newcommand{\Supp}{\normalfont\text{Supp}}
\newcommand{\Ass}{{\normalfont\text{Ass}}}
\newcommand{\Assh}{{\normalfont\text{Assh}}}

\newcommand{\Min}{{\normalfont\text{Min}}}

\newcommand{\Hom}{\normalfont\text{Hom}}

\newcommand{\FF}{\normalfont\mathcal{F}}

\newcommand{\Spec}{\normalfont\text{Spec}}

\def\f0{\mathbf{0}}

\def\1{\mathbf{1}}

\newtheorem{theorem}{Theorem}[section]

\newaliascnt{headcor}{headthm}

\aliascntresetthe{headcor}

\newaliascnt{headconj}{headthm}

\aliascntresetthe{headconj}

\newaliascnt{corollary}{theorem}
\newtheorem{corollary}[corollary]{Corollary}
\aliascntresetthe{corollary}

\newaliascnt{claim}{theorem}

\aliascntresetthe{claim}

\newaliascnt{lemma}{theorem}
\newtheorem{lemma}[lemma]{Lemma}
\aliascntresetthe{lemma}

\newaliascnt{conjecture}{theorem}
\newtheorem{conjecture}[conjecture]{Conjecture}
\aliascntresetthe{conjecture}

\newaliascnt{proposition}{theorem}
\newtheorem{proposition}[proposition]{Proposition}
\aliascntresetthe{proposition}

\theoremstyle{definition}
\newaliascnt{definition}{theorem}
\newtheorem{definition}[definition]{Definition}
\aliascntresetthe{definition}

\newaliascnt{notation}{theorem}

\aliascntresetthe{notation}

\newaliascnt{condition}{theorem}

\aliascntresetthe{condition}

\newaliascnt{example}{theorem}
\newtheorem{example}[example]{Example}
\aliascntresetthe{example}

\newaliascnt{examples}{theorem}

\aliascntresetthe{examples}

\newaliascnt{remark}{theorem}
\newtheorem{remark}[remark]{Remark}
\aliascntresetthe{remark}

\newaliascnt{question}{theorem}

\aliascntresetthe{question}

\newaliascnt{questions}{theorem}

\aliascntresetthe{questions}

\newaliascnt{problem}{theorem}

\aliascntresetthe{problem}

\newaliascnt{construction}{theorem}

\aliascntresetthe{construction}

\newaliascnt{setup}{theorem}
\newtheorem{setup}[setup]{Setup}
\aliascntresetthe{setup}

\newaliascnt{algorithm}{theorem}

\aliascntresetthe{algorithm}

\newaliascnt{observation}{theorem}

\aliascntresetthe{observation}

\newaliascnt{defprop}{theorem}

\aliascntresetthe{defprop}

\newaliascnt{fact}{theorem}
\newtheorem{fact}[fact]{Fact}
\aliascntresetthe{fact}

\def\equationautorefname~#1\null{(#1)\null}
\def\sectionautorefname~#1\null{Section #1\null}
\def\subsectionautorefname~#1\null{\S #1\null}

\usepackage{xcolor}
\definecolor{citepink}{HTML}{AA3377}

\usepackage[
  colorlinks=true,
  citecolor=citepink,
  linkcolor=blue,
  urlcolor=blue
]{hyperref}

\newcommand{\fkm}{\mathfrak{m}}

\newcommand{\fkp}{\mathfrak{p}}
\newcommand{\fkq}{\mathfrak{q}}

\def\Ass{\operatorname{Ass}}
\def\Ann{\operatorname{Ann}}
\def\depth{\operatorname{depth}}
\def\Ext{\operatorname{Ext}}
\def\grade{\mathrm{grade}}
\def\height{\mathrm{ht}}
\def\Hom{\operatorname{Hom}}
\def\Im{\mathrm{Im}}
\def\Ker{\mathrm{Ker}}
\def\m{\mathfrak{m}}

\def\Spec{\operatorname{Spec}}
\def\Supp{\operatorname{Supp}}

\title[Canonical Traces of graded Fiber Products: Applications to Stanley--Reisner rings]{Canonical Traces of Graded Fiber Products: Applications to Disconnected Stanley--Reisner Rings}

\author{Shinya Kumashiro}
\address[Kumashiro]{Department of Mathematics, Osaka Institute of Technology, 5-16-1 Omiya, asahi-ku, Osaka, 535-8585, Japan}
\email{shinya.kumashiro@oit.ac.jp}
\email{shinyakumashiro@gmail.com}

\author{Sora Miyashita}
\address[Miyashita]
{Department of Pure And Applied Mathematics, Graduate School Of Information Science And Technology, Osaka University, Suita, Osaka 565-0871, Japan}
\email{u804642k@ecs.osaka-u.ac.jp}

\date{\today}

\begin{document}

	\maketitle

\begin{abstract}
Recent work by Miyashita and Varbaro classified the canonical traces of Stanley–Reisner rings that are Gorenstein on the punctured spectrum, under the Cohen--Macaulay assumption. The purpose of this paper is to generalize the result to the non--Cohen--Macaulay case. 
First, we establish an explicit formula for the canonical trace of graded fiber products of Noetherian rings and apply it to Stanley--Reisner rings of disconnected simplicial complexes. This allows us to reduce the problem to the case of connected simplicial complexes. In that case, we succeeded in giving a complete classification without assuming the Cohen-Macaulay property. Finally, we combine these results to obtain a classification for disconnected simplicial complexes, complementing the work of Miyashita and Varbaro.
 \end{abstract}

\section{Introduction}
Let \( R \) be a Noetherian positively graded ring with a graded canonical module \( \omega_R \).  
The ideal
\[
\operatorname{tr}_R(\omega_R) := \sum_{\phi \in \operatorname{Hom}_R(\omega_R, R)} \phi(\omega_R)
\]
is called the \emph{canonical trace} of \( R \).
When $R$ is unmixed, we see that the canonical trace $\tr_R(\omega_R)$ describes the non-quasi-Gorenstein locus of $R$ (see \autoref{rem:gradedaoyamagoto}). In particular, $\tr_R(\omega_R)$ describes the non-Gorenstein locus, provided $R$ is Cohen-Macaulay. This observation has sparked significant interest in recent years, leading to an active line of research on the canonical trace, especially in the Cohen–Macaulay case  (\cite{miyazaki2021Gorenstein,hall2023nearly,dao2020trace,miyashita2024linear,herzog2019trace,celikbas2023traces,ficarra2024canonical,ficarra2024canonical!, caminata2021nearly,kumashiro2025nearly,bagherpoor2023trace,lu2024chain, lyle2024annihilators,miyazaki2024non,jafari2024nearly,moscariello2025nearly, miyashita2024canonical,miyashita2025pseudo,kimura2025trace}). 

In particular, very recently, \cite[Theorem~A]{miyashita2024canonical} classified canonical traces of Stanley--Reisner rings that are Gorenstein on the punctured spectrum, but crucially under the assumption that $R$ is Cohen–Macaulay. Let $\mm_R$ denote the graded maximal ideal of $R$. 

\begin{fact}(\cite[Theorem~A]{miyashita2024canonical})
Let $\Delta$ be a simplicial complex and $R=\kk[\Delta]$ the Stanley--Reisner ring of $\Delta$ over a field $\kk$. 
Assume that $R$ is Cohen-Macaulay. Then the following hold:
\begin{itemize}
\item[\rm (1)]
$R$ is Gorenstein on the punctured spectrum
if and only if 
$\tr_R(\omega_R) = \MM_R^i$ for some $i \in \{0,1,2\}$;
\item[\rm (2)]
$\tr_R(\omega_{R})=\MM_R$
if and only if
$\Delta$ is isomorphic either to a disjoint union of $n\ge 3$ vertices or to a path of length $n \geq 3$;
\item[\rm (3)]
$\tr_R(\omega_R)=\mm_R^2$ 
if and only if
$\Delta$ is a non-orientable $\kk$-homology manifold.
\end{itemize}
\end{fact}

The aim of this paper is to remove the assumption that ``$R$ is Cohen-Macaulay'' from \cite[Theorem~A]{miyashita2024canonical}. Of course, as mentioned before, the canonical trace in the non--Cohen--Macaulay case is likewise important for describing the non--quasi--Gorenstein locus. 
One of the main difficulties in removing the Cohen–Macaulay condition from \cite[Theorem~A]{miyashita2024canonical} arises from disconnected simplicial complexes. 
Indeed, a typical example of a non--Cohen--Macaulay Stanley--Reisner ring is one arising from a disconnected simplicial complex (in fact, any disconnected simplicial complex of dimension at least one yields a non–Cohen–Macaulay ring). 
On the other hand, such a Stanley--Reisner ring is isomorphic to the fiber product of the Stanley--Reisner rings of the connected components of the given simplicial complex (see \autoref{lem:SRfiberproduct}).

With these observations in mind, our first main theorem computes the canonical trace for general fiber products, going beyond the setting of Stanley--Reisner rings.
Let $A$, $B$ be positively graded Noetherian rings with a field $\kk=A_0=B_0$. Let $f:A \to \kk$ and $g: B \to \kk$ be natural graded projections of graded rings. Then, 
$$
R := A \times_\kk B = \{(a, b) \in A \times B : f(a) = g(b)\}
$$
is called the fiber product of $A$ and $B$ over $\kk$. 
For simplicity, we define 
\[
\tr_R^\dagger(\omega_R) :=
\begin{cases}
\tr_R(\omega_R) & \text{if } R \text{ is not quasi-Gorenstein}, \\
\mm_R & \text{if } R \text{ is quasi-Gorenstein,}
\end{cases}
\]
(see \autoref{def:dagertrace} and \autoref{def:daggertukaimasita}). 
\( (0) :_R \mathfrak{m}_R \) denotes the annihilator of \( \mathfrak{m}_R \) in \( R \). 
With the notation, our first main result in this paper is stated as follows:

\begin{theorem}[{\autoref{aaa}}]\label{THM:nice!}
Assume that $A \neq A_0$,
$B \neq B_0$ and $\dim(R)\ne 1$.
Then 
\[
\tr_R(\omega_R) = \begin{cases}
    \tr_A^\dagger(\omega_{A})R \oplus \tr_B^\dagger(\omega_{B})R & \text{if } \dim(A)=\dim(B) \\
    \tr_A^\dagger(\omega_{A})R \oplus ((0):_B \MM_B)R & \text{if } \dim(A)>\dim(B) \\    
    ((0):_A \MM_A)R \oplus \tr_B^\dagger(\omega_{B})R & \text{if } \dim(A)<\dim(B).
\end{cases}
\]
\end{theorem}

\autoref{THM:nice!} fails in general in Krull dimension one~(see \autoref{rem:OK...}); nevertheless, we conjecture that, under appropriate hypotheses, one can extend  \autoref{THM:nice!} to dimension one (see \autoref{conjX}).


We now return to Stanley--Reisner rings. By \autoref{THM:nice!}, it is enough to consider connected simplicial complexes. The question of whether every nearly Gorenstein Stanley--Reisner ring of dimension at least three is Gorenstein, originally posed by the second author in \cite[Section~4]{miyashita2024levelness}, was one of the main motivations for \cite{miyashita2024canonical}; it was answered affirmatively in \cite[Corollary~3.5]{miyashita2024canonical}.
The result below removes the Cohen--Macaulay assumption from this statement, giving a complete generalization of \cite[Corollary~3.5]{miyashita2024canonical} to the non-Cohen--Macaulay setting; it may also be viewed as a non-Cohen--Macaulay analogue of \cite[Theorem~A]{miyashita2024canonical}.

\begin{theorem}[{see \autoref{thm:canonicalsquare}}]\label{thm:canonicalsquareee}
Let $\Delta$ be a connected simplicial complex and set $R=\kk[\Delta]$.
Then the following hold:
\begin{itemize}
\item[\rm (1)]
Assume that $R$ is Cohen--Macaulay on the punctured spectrum.
Then
$\sqrt{\tr_R(\omega_R)}=\mm_R$
if and only if 
$\tr_R(\omega_R) = \MM_R^i$ for some $i \in \{0,1,2\}$;
\item[\rm (2)]
$\tr_R(\omega_{R})=\MM_R$
if and only if
$\Delta$ is isomorphic to a path of length $n \geq 3$;
\item[\rm (3)]
Assume that $R$ is Cohen--Macaulay on the punctured spectrum.
Then
$\tr_R(\omega_R)=\mm_R^2$ 
if and only if
$\Delta$ is a non-orientable $\kk$-homology manifold.
\end{itemize}
In particular, if $\tr_R(\omega_R) \supseteq \mm_R$
and $\dim(R) \ge 3$,
then $R$ is quasi-Gorenstein.
\end{theorem}

Combining \autoref{THM:nice!} with \autoref{thm:canonicalsquareee}, we finally obtain the second main result in this paper, \autoref{MainTHM:C} (it should be noted that \autoref{thm:canonicalsquareee} does not exactly coincide with the assertion obtained by formally substituting $n=1$ into \autoref{MainTHM:C}). 
For terminology not defined, see \autoref{stanleyNICE}.

\begin{theorem}[{\autoref{thm:interesting???}}]\label{MainTHM:C}
Fix $2 \le n \in \ZZ$.
For \( 1 \leq i \leq n \),
let $\Delta_i$ be a connected simplicial complex and let $A_i:=\kk[\Delta_i]$.
Set $\Delta:=\bigsqcup_{i=1}^n \Delta_i$~(see \autoref{def:disjunionsimpcomp}) and
$R:=\kk[\Delta]$, and assume that $\Delta$ is not the discrete simplicial complex on two vertices.
Then the  following hold:
\begin{itemize}
\item[\rm (1)] Suppose that $R$ is Cohen--Macaulay on the punctured spectrum.
Then the following are equivalent:
\begin{itemize}
\item[\rm (a)] $\tr_R(\omega_R)$ is $\MM_R$-primary;
\item[\rm (b)] $\tr_{A_i}(\omega_{A_i}) \in \{A_i, \mm_{A_i}, \mm_{A_i}^2\}$ and \( \dim(\Delta_i) = \dim(\Delta) \) for any $i=1,\cdots,n$;
\item[\rm (c)] $\tr_R(\omega_R) \supseteq \MM_R^2$.
\end{itemize}
\item[\rm (2)] The following are equivalent:
\begin{itemize}
\item[\rm (a)] $\tr_R(\omega_R)=\MM_R$;
\item[\rm (b)] $\tr_{A_i}(\omega_{A_i}) \supseteq \MM_{A_i}$ and $\dim(\Delta_i)=\dim(\Delta)$ for
any $1 \le i \le n$;
\item[\rm (c)]
The following hold;
\begin{itemize}
\item[\rm (i)] $\dim(\Delta_i)=\dim(\Delta)$ for any $1 \le i \le n$,
\item[\rm (ii)] \( A_i \) is quasi-Gorenstein or \( \Delta_i \) is isomorphic to a path for any $1 \le i \le n$.
\end{itemize}
\end{itemize}
\item[\rm (3)] 
Suppose that $R$ is Cohen--Macaulay on the punctured spectrum.
Then the following are equivalent:
\begin{itemize}
\item[\rm (a)] $\tr_R(\omega_R)=\MM_R^2$;
\item[\rm (b)] $\tr_{A_i}(\omega_{A_i})=\MM_{A_i}^2$ and \( \dim(\Delta_i) = \dim(\Delta) \) for any \( 1 \leq i \leq n \);
\item[\rm (c)] $\Delta_i$ is a $\kk$-non-orientable $\kk$-homology manifold and \( \dim(\Delta_i) = \dim(\Delta) \) for any \( 1 \leq i \leq n \).
\end{itemize}
\end{itemize}
\end{theorem}

\subsection*{Outline}
In \autoref{sect2}, we review the fundamental concepts of trace ideals, canonical modules, fiber products, and Stanley--Reisner rings.
In \autoref{sec:trfib}, we discuss the canonical trace of the fiber product and prove \autoref{THM:nice!}.
By applying the result
to Stanley--Reisner rings, we also establish \autoref{MainTHM:B}, a formula for computing the canonical traces of Stanley--Reisner rings arising from disconnected simplicial complexes.
In \autoref{sec:teterStanley}, we focus on the case of Stanley--Reisner rings corresponding to disconnected simplicial complexes. We extend known results on rings called {\it Teter type} in the Cohen--Macaulay setting to the non-Cohen--Macaulay case~(see \autoref{Thm:nonCMTeter}), and we apply it to prove \autoref{thm:canonicalsquareee}. Finally, by combining \autoref{thm:canonicalsquareee} with \autoref{MainTHM:B}, we prove \autoref{MainTHM:C}, which generalizes \cite[Theorem~A]{miyashita2024canonical}.

\begin{setup}\label{setup1}
Throughout this paper,
we denote the set of non-negative integers by $\NN$. Let $R=\bigoplus_{i \ge 0} R_i$ be a positively graded Noetherian ring. Unless otherwise stated, we assume that $(R_0, \fkm_{R_0})$ is an Artinian local ring. Hence, $R$ has the unique graded maximal ideal given by ${\MM_R}:=\fkm_{R_0} R + \bigoplus_{i > 0} R_i$.
When there is no risk of confusion about \(R\), we simply write \(\mathfrak{m}_R\) as \(\mathfrak{m}\).
We denote the residue field by $\kk=R/\MM_R$. For graded $R$-modules $M$ and $N$, ${}^*\Hom_R(M, N)$ denotes the graded $R$-module consisting of graded homomorphisms from $M$ to $N$.
Let $\omega_R$ denote the graded canonical module (note that under these assumptions, $R$ admits a graded canonical module; see \autoref{defcanon}). Let $a_R$ denote the {\it $a$-invariant} of $R$, that is, 
\[
a_R := - \min \{j\in \mathbb{Z} : [\omega_R]_j \neq 0 \}.
\]
\end{setup}

\begin{remark}\label{rem1.5}
Under the assumption of \autoref{setup1}, there exists a graded polynomial ring $S$ such that $S_0$ is a regular local ring and there is a surjective graded homomorphism $S\to R$. Indeed, since an Artinian local ring $R_0$ is complete, by the structure theorem of complete local ring, there exists a regular local ring $S_0$ mapping onto $R_0$. Hence, we can find a graded polynomial ring $S$ over $S_0$ such that there is a surjective graded homomorphism $S\to R$. 
\end{remark}

\section{Preliminaries}
\label{sect2}
The purpose of this section is to lay the groundwork for the discussions of our main results. Throughout this section, unless otherwise stated, we maintain \autoref{setup1}.

\subsection{Trace ideals}

\begin{definition}
For a graded $R$-module $M$, the sum of all images of homomorphisms $\phi \in \Hom_R(M,R)$ is called the {\it trace} of $M$:
\[
\tr_R(M):=\sum_{\phi \in \Hom_R(M,R)}\phi(M).
\] 
\end{definition}

\begin{remark}\label{rem:interestingfiniteness}
Let $M$ be a (not necessarily finitely generated) graded $R$-module $M$. Then, we have 
\[
\tr_R(M)=\sum_{\phi \in {}^*\Hom_R(M,R)}\phi(M).
\] 
\end{remark}

\begin{proof}
To prove the inclusion $\subseteq$, we prove that $f(M)\subseteq \sum_{\phi \in {}^*\Hom_R(M,R)}\phi(M)$ for all $f\in \Hom_R(M,R)$. Since $M$ is graded, it is enough to prove that 
\[
f(x) \in \sum_{\phi \in {}^*\Hom_R(M,R)}\phi(M)
\]
for all $f\in \Hom_R(M,R)$ and all homogeneous elements $x\in M$. We write 
\[
f(x)=\sum_{i\in \mathbb{Z}, \, \text{finite sum}} y_i
\]
for $y_i\in R_i$. We then consider the composition $g_{n,n+i}: M_n\xrightarrow{\iota} M\xrightarrow{f} R \xrightarrow{\pi} R_{n+i}\xrightarrow{\iota} R$ for integers $i$ and $n$, where $\iota$ are the inclusions and $\pi$ is the surjection. Set a homogeneous element 
\[
g_{i}:=(g_{n,n+i})_{n\in \mathbb{Z}}\in {}^*\Hom_R(M,R)
\]
of degree $i$. By the definition of $g_{i}$, we have $y_i=g_{i-\deg x}(x)$. Hence, $f(x)$ can be represented by the finite sum of $g_{i}$:
\[
f(x) = \sum_{i\in \mathbb{Z}, \, \text{finite sum}} y_i = \sum_{i\in \mathbb{Z}, \, \text{finite sum}} g_{i-\deg x}(x).
\]
It follows that $f(x)\in \sum_{\phi \in {}^*\Hom_R(M,R)}\phi(M)$. 
The converse inclusion $\supseteq$ is clear.
\end{proof}

The following \autoref{lem:thanks} and \autoref{tracelem}, which concern trace ideals, will be used in Section~\ref{sec:trfib}.
\autoref{lem:thanks}~(3) is originally stated for the local case (see, for example,
    \cite[Lemma 3.1]{kumashiro2023trace}), but the proof can be adapted to the graded case with only minor modifications. For the reader's convenience, we include a proof here.

\begin{lemma}\label{lem:thanks}
Let $M$ be a non-zero $R$-module.
Then the following hold:
\begin{itemize}
\item[\rm (1)] Let $N$ be a graded $R$-module such that $\MM_R \subseteq \ann_R(N)$.
Then there exists a set $\Lambda$ and an epimorphism $f: M^{\bigoplus \Lambda} \twoheadrightarrow N$;
\item[\rm (2)] Let $I$ be a graded ideal of $R$ such that 
$\MM_R \subseteq \ann_R(I)$.
Then we have $\tr_R(M) \supseteq I$;
\item[\rm (3)] We have $\tr_R(M) \supseteq (0):_R \mm_R$.
\end{itemize}
\begin{proof}
(1):
Since $N$ is a \( \kk \)-vector space, there exists an index set \( \Lambda \) such that \( N \cong \kk^{\bigoplus \Lambda} \).  
On the other hand, noting that \( M / \mm_R M \) is a non-zero \( \kk \)-vector space by Nakayama's Lemma, we have an \( R \)-epimorphism
$M \twoheadrightarrow M / \mm_R M \twoheadrightarrow \kk$.  
It follows that there exists an \( R \)-epimorphism $M^{\bigoplus \Lambda} \twoheadrightarrow N$.

(2): Apply (1) by taking \( N = I \).  
Noting \cite[Proposition~2.8~(i)]{lindo2017trace}, we obtain  
$\tr_R(M) \supseteq \tr_R(I) \supseteq I$,  
as desired.

(3): This follows from (2).
\end{proof}
\end{lemma}

\begin{lemma}\label{tracelem}
Let $M$ be a finitely generated graded $R$-module. Suppose that $\tr_R(M) \subseteq \fkm_R$ (equivalently, $M$ has no free summand). Let $N$ be an $R$-submodule of $M$ such that $N \subseteq ((0):_R \fkm_R)M$. Then we have $\tr_R(M) = \tr_R(M/N)$.
\end{lemma}
\begin{proof}
Let $f\in \Hom_R(M, R)$. Since $f(M)\subseteq \tr_R(M) \subseteq\fkm_R$, we get 
\[
f(N)\subseteq f(((0):_R \fkm_R) M)= ((0):_R \fkm_R)f(M)\subseteq ((0):_R \fkm_R)\fkm_R =0.
\]
Hence, $f$ induces the canonical map $\phi_f: M/N \to R$, and $\phi_f(M/N)=f(M)$. Thus, $\tr_R(M) \subseteq\tr_R(M/N)$. The converse inclusion follows from the fact that there is a surjection $M\to M/N$ (see, for example, \cite[Proposition 2.8~(i)]{lindo2017trace}).
\end{proof}

\subsection{Canonical modules over Noetherian graded rings}

Let us recall the definition of the canonical module over a positively graded Noetherian ring.
We maintain \autoref{setup1}.

\begin{definition}{\cite[Definition (2.1.2)]{goto1978graded}}\label{defcanon}
Set $d=\dim (R)$. The finitely generated graded $R$-module
\[
\omega_R:={}^*\Hom_R(\mathrm{H}_{\MM_R}^{d}(R), E_R)
\]
is called the {\it canonical module}, where $\mathrm{H}_{\MM_R}^{d}(R)$ denotes the $d$-th local cohomology and $E_R$ denotes the injective envelope of $R/\mm_R$.
\end{definition}

\begin{definition}
\( R \) is said to be {\it quasi-Gorenstein} if \( \omega_R \) is isomorphic to \( R \) as a graded \( R \)-module.  
\end{definition}

\begin{lemma}{\cite[Proposition (2.1.6)]{goto1978graded}}\label{gotowatanabe}
Let $S$ be a positively graded Gorenstein local ring such that there exists a graded surjective ring homomorphism $S\to R$. Then 
\[
\omega_R\cong ^{*}\Ext_S^t(R, S),
\]
where $t=\dim(S)-\dim(R)$ which is the least integer $i$ such that $^{*}\Ext_S^i(R, S)\ne 0$.
\end{lemma}


\begin{definition}
{Let \( M \) be a graded \( R \)-module. Let \( S \subseteq R \) be a multiplicatively closed subset of \( R \), that is,
\( 1_R \in S \) and
\( ab \in S \) for any \( a, b \in S \).  
We denote by \( (S) \) the set of all homogeneous elements of \( S \).
We denote \( M_{(S)} \) by the localization of \( M \) at \( (S) \).
For any \( \mathfrak{p} \in {}^{*}\Spec(R) \), we define \( R_{(\mathfrak{p})} \) as \( R_{(R \setminus \mathfrak{p})} \).}
\end{definition}

\begin{remark}\label{rem:gradedaoyamagoto}
The following hold:
\begin{enumerate}[\rm(1)]
\item
We have $(\omega_R)_\pp \cong \omega_{R_\pp}$ and $(\omega_R)_{(\pp)} \cong \omega_{R_{(\pp)}}$ for any $\pp \in {}^*\Supp(\omega_R)$.
\item $\Ass_R(\omega_R)=\{\pp \in {}^*\Spec(R) : \dim(R/\pp)=\dim (R) \}$.
Consequently, we have
${}^*\Supp(\omega_R)= {}^*\Spec(R)$ if and only if $\Min(R)=\{\pp \in {}^*\Spec(R) : \dim(R/\pp)=\dim (R) \}$,
that is,
$R$ is {\it equidimensional}.
\item
Let
$(0) = \bigcap_{\mathfrak{p} \in \mathrm{Ass}(R)} Q(\mathfrak{p})$
be an irredundant primary decomposition of the zero ideal \((0)\), where each \(Q(\mathfrak{p})\) is a primary ideal with \(\sqrt{Q(\mathfrak{p})} = \mathfrak{p}\). Then we have
\[
\operatorname{ann}_R(\omega_R) = \bigcap_{\substack{\mathfrak{p} \in \mathrm{Ass}(R) \\ \dim(R) = \dim(R/\mathfrak{p})}} Q(\mathfrak{p}).
\]
\item $\tr_R(\omega_R)=R$ if and only if $R$ is quasi-Gorenstein.
\item Let $\pp \in {}^*\Supp(\omega_R)$. Then $\pp \supseteq \tr_R(\omega_R)$ if and only if $R_\pp$ is not quasi-Gorenstein.
\item If $\sqrt{\tr_R(\omega_R)} \supseteq \mm_R$,
then $R$ is equidimensional.
Consequently, \(R\) is quasi-Gorenstein on the punctured spectrum.
\end{enumerate}
\end{remark}
\begin{proof}
(1): Let $\phi:S\to R$ be a surjective graded ring homomorphism, where $S$ is a positively graded Gorenstein ring~(see \autoref{rem1.5}). Set $t=\dim(S)-\dim(R)$. By \cite[1.2.10 (e)]{bruns1998cohen}, we can choose a graded $S$-regular sequence $\underline{x}=x_1, \dots, x_t \in \Ker(\phi)$. 
By \cite[Lemma 3.1.16]{bruns1998cohen}, passing to $S\to S/(\underline{x})$, we may assume that $t=0$. 

Let $\pp \in {}^*\Supp(\omega_R)$. Then, we have $(\omega_R)_\fkp\cong \Hom_S(R, S)_\fkp$ by \autoref{gotowatanabe}. On the other hand, we have $\omega_{R_\fkp} \cong \Hom_{S_\fkq}(R_\fkp, S_\fkq)$, where $\fkq=\fkp\cap S$, since $S_\fkq \to R_\fkp$ is surjective.
Therefore, we have $(\omega_R)_\pp
\cong
(\Hom_S(R,S))_\pp
\cong
\Hom_{S_\fkq}(R_\pp,S_{\fkq})
\cong
\omega_{R_\fkp}$.
Similarly, since
\autoref{gotowatanabe} and \( S_{(\fkq)} \to R_{(\fkp)} \)  
is surjective, we have  
\((\omega_R)_{(\pp)} \cong \Hom_S(R, S)_{(\pp)}\),  
and  
\(\omega_{R_{(\fkp)}} \cong \Hom_{S_{(\fkq)}}(R_{(\fkp)}, S_{(\fkq)})\),
so it follows that  
\((\omega_R)_{(\pp)} \cong \omega_{R_{(\pp)}}\).

(2): 
Let $\fkp\in {}^*\Spec(R)$. Since all associated primes are graded,  $\fkp\in \Ass(\omega_R)$ if and only if $\fkp R_\fkm\in \Ass(\omega_{R_\fkm})$ (see \cite[Theorem 6.2]{matsumura1989commutative}). On the other hand, we obtain that 
\begin{align*}
\Ass(\omega_{R})_\fkm = \Ass(\omega_{R_\fkm}) =& \{\fkq \in \Spec(R_\fkm) : \dim(R_\fkm/\fkq)=\dim (R_\fkm) \}\\
=& \{\pp R_\fkm \in \Spec(R_\fkm) : \pp \in {}^*\Spec(R) \text{ and } \dim(R_\fkm/\pp R_\fkm)=\dim (R_\fkm) \} \\
=& \{\pp R_\fkm \in \Spec(R_\fkm) : \pp \in {}^*\Spec(R) \text{ and } \dim(R/\fkp)=\dim (R) \}, 
\end{align*}
where the first equation follows by (1), the second equation follows by \cite[(1.7)]{aoyama1983some}, the third and the fourth equations follow by \cite[Theorem 1.5.8 (a) and (b)]{bruns1998cohen}. Therefore, we have $\fkp\in \Ass(\omega_R)$ if and only if
$\fkp \in \{\fkq \in {}^*\Spec(R) : \dim(R/\fkq)=\dim (R) \}$, as desired.

(3):
By \cite[Theorem~6.2]{matsumura1989commutative}, we have
\[
\{ \mathfrak{q} \in \mathrm{Ass}(R_{\mathfrak{m}}) : \dim(R_{\mathfrak{m}}/\mathfrak{q}) = \dim(R_{\mathfrak{m}}) \} = \{ \mathfrak{p} R_{\mathfrak{m}} : \mathfrak{p} \in \mathrm{Ass}(R),\ \dim(R/\mathfrak{p}) = \dim(R) \}.
\]
Thus we obtain an irredundant primary decomposition of \((0)\) in \(R_{\mathfrak{m}}\) as \((0) = \bigcap_{\mathfrak{p} \in \mathrm{Ass}(R)} Q(\mathfrak{p} R_{\mathfrak{m}})\). By \cite[(1.8)]{aoyama1983some}, it follows that
\[
\operatorname{ann}_{R_{\mathfrak{m}}}(\omega_{R_{\mathfrak{m}}}) = \bigcap_{\substack{\mathfrak{p} \in \mathrm{Ass}(R) \\ \dim(R) = \dim(R/\mathfrak{p})}} Q(\mathfrak{p} R_{\mathfrak{m}}).
\]
Note that \(\operatorname{ann}_{R_{\mathfrak{m}}}(\omega_{R_{\mathfrak{m}}})=\operatorname{ann}_R(\omega_R) R_{\mathfrak{m}}\) by (1) and the compatibility of annihilators with localization. Therefore, we have
\[
\operatorname{ann}_R(\omega_R) R_{\mathfrak{m}} = \operatorname{ann}_{R_{\mathfrak{m}}}(\omega_{R_{\mathfrak{m}}}) = \bigcap_{\substack{\mathfrak{p} \in \mathrm{Ass}(R) \\ \dim(R) = \dim(R/\mathfrak{p})}} Q(\mathfrak{p} R_{\mathfrak{m}}) = \left( \bigcap_{\substack{\mathfrak{p} \in \mathrm{Ass}(R) \\ \dim(R) = \dim(R/\mathfrak{p})}} Q(\mathfrak{p}) \right) R_{\mathfrak{m}},
\]
and hence
\[
\operatorname{ann}_R(\omega_R) = \left( \operatorname{ann}_R(\omega_R) R_{\mathfrak{m}} \right) \cap R = \left( \bigcap_{\substack{\mathfrak{p} \in \mathrm{Ass}(R) \\ \dim(R) = \dim(R/\mathfrak{p})}} Q(\mathfrak{p}) \right) R_{\mathfrak{m}} \cap R = \bigcap_{\substack{\mathfrak{p} \in \mathrm{Ass}(R) \\ \dim(R) = \dim(R/\mathfrak{p})}} Q(\mathfrak{p}).
\]

(4): Suppose that $\tr_R(\omega_R)=R$. Then, $\omega_R \cong R\oplus X$ as graded modules for some graded $R$-module $X$. On the other hand, by \cite[Proposition 2.8~(viii)]{lindo2017trace}, $\tr_{R_\fkm}(\omega_{R_\fkm})=\tr_R(\omega_R)R_\fkm=R_\fkm$. It follows that $R_\fkm$ is quasi-Gorenstein by \cite[Proposition 3.3]{aoyama1985endomorphism}, that is, $\omega_{R_\fkm}\cong R_\fkm$. Therefore, we observe that $X_\fkm=0$. Since $X$ is graded, this shows that $X=0$. Thus, $\omega_R\cong R$, that is, $R$ is quasi-Gorenstein. The converse is clear.

(5): 
Notice that $\tr_{R_p}(\omega_{R_\pp})=\tr_R(\omega_R)R_\pp$ by \cite[Proposition 2.8~(viii)]{lindo2017trace} and (1).
Then we observe that
$$\pp \supseteq \tr_R(\omega_R)
\iff
\tr_R(\omega_R)R_\pp \neq R_\pp
\iff
\tr_{R_p}(\omega_{R_\pp}) \neq R_\pp.$$
Moreover, $\tr_{R_p}(\omega_{R_\pp}) \neq R_\pp$ is equivalent to that $R_\pp$ is not quasi-Gorenstein by \cite[Proposition 3.3]{aoyama1985endomorphism}.

(6):
By (2), it is enough to show that ${}^*\operatorname{Supp}(\omega_R)={}^*\operatorname{Spec}(R)$.
Assume that \(
\sqrt{\operatorname{tr}_R(\omega_R)}
\supseteq
\mathfrak m_R
\). Let \(\mathfrak p\in {}^*\operatorname{Spec} R\). The case \(\mathfrak p=\mathfrak m_R\) is clear, so we may assume that \(\mathfrak p\neq\mathfrak m_R\).
Then,
we have
\(\tr_{R_{\mathfrak p}}((\omega_R)_{\mathfrak p})=\tr_R(\omega_R)_{\mathfrak p}=R_{\mathfrak p}\)
by \cite[Proposition 2.8~(viii)]{lindo2017trace}. In particular,
we obtain
\((\omega_R)_{\mathfrak p}\neq 0\). Hence \(\mathfrak p\in{}^*\Supp(\omega_R)\), and therefore \(R\) is equidimensional. The assertion now follows from \rm(5), which shows that \(R\) is quasi-Gorenstein on the punctured spectrum.
\end{proof}

\begin{definition}
$R$ is said to be {\it generically Gorenstein}
if $R_\pp$ is Gorenstein for any $\pp \in \Ass(R)$.
\end{definition}

\begin{definition}
Let \( S \) be the set of all non-zero divisors of \( R \) and define \( {}^*Q(R) := R_{(S)} \).
\end{definition}

\begin{remark}\label{rem:genGoriso}
If $R$ is generically Gorenstein, 
then ${}^*Q(R)$ is Gorenstein.
\begin{proof}
It follows from \cite[Exercises 3.6.20~(b)]{bruns1998cohen}.
\end{proof}
\end{remark}

\begin{remark}\label{rem:Canon.NZD}
Suppose that $R$ is generically Gorenstein.
Then the following hold:
\begin{itemize}
\item[\rm (1)] $\omega_R$ is isomorphic to a graded ideal of $R$.
\item[\rm (2)]
Let $I_R$ denote a graded ideal which is isomorphic to $\omega_R$.
Then the following are equivalent:
\begin{itemize}
\item[\rm (a)] $R$ is {\it unmixed}, that is, $\dim(R)=\dim(R/\pp)$ for any $\pp \in \Ass(R)$;
\item[\rm (b)] $\ann_R(\omega_R)=(0)$;
\item[\rm (c)] $\grade(I_R)>0$.
\end{itemize}
\end{itemize}
\end{remark}
\begin{proof}
(1):
Since \( R \) is generically Gorenstein,
there exists a graded isomorphism $\psi: {}^*Q(R) \otimes_R \omega_R \rightarrow {}^*Q(R)$ by \autoref{rem:genGoriso}.
Consider the graded \( R \)-homomorphism
\[
\phi:
\;
\omega_R \rightarrow {}^*Q(R) \otimes_R \omega_R,
\quad x \mapsto 1 \otimes x,
\]
which is injective because \( \Ass(\omega_R) \subseteq \Ass(R) \) by \autoref{rem:gradedaoyamagoto}~(2).  
Then we have \( \omega_R \cong \operatorname{Im}(\psi \phi) \subseteq {}^*Q(R) \).  
Let \( f_1, \dots, f_r \) be graded minimal generators of \( \operatorname{Im}(\phi) \) as a graded \( R \)-module.  
For each \( 1 \le i \le r \), we can write  
$f_i =
\frac{a_{i}}{s_{i}}
$,
where \( a_{i} \in R \) and each \( s_{i} \in R \) is a homogeneous non-zero divisor.  
Thus we obtain a graded embedding  
$\omega_R \cong (\prod_{i=1}^n s_i) \Im(\phi) \hookrightarrow R$.

(2):
(a) $\Leftrightarrow$ (b) follows from \autoref{rem:gradedaoyamagoto}~(2) and (3).
We now prove (b) $\Leftrightarrow$ (c).
If (b) holds, then since \( \ann_R(\omega_R) = \ann_R(I_R) \cong {}^*\Hom_R(R/I_R, R) = 0 \) and 
$\grade(I_R) = \min \{ i : {}^*\Ext_R^i(R/I_R, R) \neq 0 \} > 0$,  
it follows that (c) holds.
If (c) holds, then \( I_R \) contains a non-zero divisor of \( R \), which implies (b).
\end{proof}

\begin{remark}\label{rem:QGPolyExtension}
The following are equivalent:
\begin{itemize}
\item[\rm (1)] $R$ is quasi-Gorenstein;
\item[\rm (2)] $R[x]$ is quasi-Gorenstein;
\item[\rm (3)] $R[x, x^{-1}]$ is quasi-Gorenstein.
\end{itemize}
\end{remark}
\begin{proof}
Since $R\to R[x]$ and $R\to R[x, x^{-1}]$ are flat ring homomorphisms, we get 
\[
R[x]\otimes_R\tr_R(\omega_R) =\tr_{R[x]}(\omega_{R[x]}) \quad \text{and} \quad R[x, x^{-1}]\otimes_R\tr_R(\omega_R) =\tr_{R[x, x^{-1}]}(\omega_{R[x, x^{-1}]})
\]
(see, for example, \cite[Proposition 2.8~(viii)]{lindo2017trace}). Hence, the assertion follows by \autoref{rem:gradedaoyamagoto}~(4).
\end{proof}

\subsection{Fiber products}\label{subsection23}
Let $A$, $B$, and $T$ be positively graded Noetherian rings. Let $f:A \to T$ and $g: B \to T$ be graded homomorphisms of graded rings. The subring of $A\times B$ 
$$
R := A \times_T B = \{(a, b) \in A \times B : f(a) = g(b)\}
$$
is called {\it the fiber product of $A$ and $B$ over $T$ with respect to $f$ and $g$}. Then, one can endow $R$ with a natural graded structure by $R_n:=\{(a,b)\in A_n\times B_n : f(a)=g(b)\}$. 
By definition of the fiber product, we get a graded exact sequence
\begin{align}\label{eq0}
0 \longrightarrow R \overset{\iota}{\longrightarrow} A \oplus B \overset{\phi}{\longrightarrow} T
\end{align}
of $R$-modules, where $\phi = \left(\begin{smallmatrix}
f \\
-g
\end{smallmatrix}\right)$.
The map $\phi$ is surjective if $f$ or $g$ is surjective. 

In what follows, we suppose the following.
\begin{setup}\label{conditiona}
 \( T = A_0 = B_0 \) and the maps \( f \) and \( g \) are the canonical graded surjections 
\begin{center}
$f:A\to A/A_{>0}\cong T$ \quad and \quad $g: B\to B/B_{>0}\cong T$.
\end{center}
\end{setup}

In later chapters, we will mainly work under the further assumption that $T$ is a field, however, we state the propositions here in as general a form as possible.

\begin{remark}\label{rem:12}
We have
$\dim (R)=\max\{\dim(A), \dim (B)\}$~(see \cite[Lemma~1.5]{endo2021almost}).
\end{remark}

\begin{lemma}\label{lem:27}
Suppose that $(T, \fkm_T)$ is a Noetherian local ring. Then, $R$ has the unique graded maximal ideal 
\[
\fkm_R=\{(a,a) : a\in \fkm_T\} \oplus (A_{>0}\times B_{>0}).
\]
\end{lemma}

\begin{proof}
Since $A_{>0} \times B_{>0} \subseteq R \subseteq A\times B$, we have $R_{>0}=A_{>0} \times B_{>0}$. We also have $\{(a,a) : a\in \fkm_T\} \subseteq R_0$ by the definitons of $f$ and $g$. Hence, $\fkm_R=\{(a,a) : a\in \fkm_T\} \oplus (A_{>0}\times B_{>0})$ is a graded ideal of $R$. Since $R/\fkm_R\cong T/\fkm_T$, $\fkm_R$ is a graded maximal ideal of $R$. Assume that there exists another graded maximal ideal $\fkn_R$, and choose $\alpha\in \fkn_R\setminus \fkm_R$. Then, $\alpha=(a, b)$ is of degree $0$; hence, $a=b\in T$. Since $\alpha\not\in \fkm_R$,
we get $a\not\in \fkm_T$.
This implies that $a$ is a unit of $T$. This shows that $\alpha=(a,a)$ is also a unit of $R$, which is a contradiction since $\alpha \in \fkn_R$.
\end{proof}

\begin{proposition}\label{lem:important1}
Suppose that $(T, \fkm_T)$ is an Artinian local ring. 
Set 
\[
\omega_{A, B}:=\begin{cases}
    \omega_{A} \oplus \omega_{B} & \text{if } \dim(A)=\dim(B) \\
    \omega_{A} & \text{if } \dim(A)>\dim(B) \\    
    \omega_{B} & \text{if } \dim(A)<\dim(B).
\end{cases}
\]
Then, the following hold. 
\begin{enumerate}[\rm(1)] 
\item If $\dim (R)=0$, then there exists a graded exact sequence 
\[
0\rightarrow \omega_T \rightarrow
\omega_{A} \oplus \omega_{B} \rightarrow \omega_R
\rightarrow 0.
\]
\item If $\dim (R)=1$, then there exists a graded exact sequence 
\[
0 \rightarrow
\omega_{A, B} \rightarrow \omega_R
\rightarrow \omega_T.
\]
\item If $\dim (R)\ge 2$, then $\omega_R \cong \omega_{A, B}$.
\end{enumerate}
\end{proposition}

\begin{proof}
By \autoref{rem1.5}, there are a graded polynomial ring $S$ over a regular local ring $S_0$ and a surjective graded homomorphism $S\to R$. Set $t=\dim(S)-\dim(R)$. By applying the functor $\Hom_S(-, S)$ to the short exact sequence \eqref{eq0} (note that $\phi$ is surjective since $f$ and $g$ are surjective), we get a graded exact sequence 
\[
\Ext_S^{t-1}(R,S) \rightarrow \Ext_S^{t}(T,S) \rightarrow
\Ext_S^{t}(A,S) \oplus \Ext_S^{t}(B,S) \rightarrow \Ext_S^{t}(R,S)
\rightarrow \Ext_S^{t+1}(T,S).
\]
By \autoref{gotowatanabe}, we have $\Ext_S^{t}(R,S) \cong \omega_R$ and $\Ext_S^{t-1}(R,S)=0$. We further have an isomorphism $\Ext_S^{t}(A,S) \oplus \Ext_S^{t}(B,S) \cong \omega_{A, B}$ by \autoref{gotowatanabe} and \autoref{rem:12}. 
Hence, the above exact sequence induces the following:
\begin{align}\label{eq1}
0\rightarrow \Ext_S^{t}(T,S) \rightarrow
\omega_{A, B} \rightarrow \omega_R
\rightarrow \Ext_S^{t+1}(T,S).
\end{align}
Therefore, noting that $T$ is Artinian and applying \autoref{gotowatanabe}, we get the assertion.
\end{proof}

For use in the next subsection, we inductively define the fiber product of $n$ rings.

\begin{definition}
Let \(A_i\) and $T$ be positively graded Noetherian rings, \( A_i \to T \) graded ring homomorphisms for $i=1, \dots, n$.
We inductively define the fiber product of \( A_1, \dots, A_n \) over \( T \) with respect to \( f_1, \dots, f_n \), that is, 
\[
A_1 \times_T \cdots \times_T A_n:= (A_1 \times_T \cdots \times_T A_{n-1})\times_T A_{n}.
\]
\end{definition}




\subsection{Stanley--Reisner rings}
Next, we summarize the terminology related to Stanley--Reisner rings that is necessary for this paper.
In particular, we introduce the fact that the Stanley-Reisner ring of a disconnected simplicial complex has the structure of a fiber product~(see \autoref{lem:SRfiberproduct}).
{We begin by stating the definition of a simplicial complex.
For a set $X$, we denote its power set by $\mathcal{P}(X)$.

\begin{definition}
{Let \( V \) be a finite set.  
A subset \( \Delta \subseteq \mathcal{P}(V) \) is said to be a \textit{simplicial complex on \( V \)} if it satisfies the following three conditions:}
\begin{enumerate}
    \item \( \Delta \neq \emptyset \);
    \item $\{ v \} \in \Delta$ for any $v \in V$;
    \item For every \( \sigma \in \Delta \) and every \( \tau \in \mathcal{P}(V) \),
    we have \( \tau \in \Delta \) if \( \tau \subseteq \sigma \).
\end{enumerate}
Given a simplicial complex \( \Delta \) on a finite set \( V \), we call \( V \) the \textit{vertex set} of \( \Delta \), and denote it by \( V(\Delta) \).
\end{definition}

For simplicity, we often omit mention of the vertex set and refer to \( \Delta \) simply as a simplicial complex.

\begin{remark}
A simplicial complex $\Delta$ can be regarded as a partially ordered set under the inclusion relation $\subseteq$.  
Given two simplicial complexes $\Delta_1$ and $\Delta_2$,  
if there exists an order-preserving isomorphism between them as partially ordered sets,  
we say that $\Delta_1$ and $\Delta_2$ are \textit{isomorphic}.  
Isomorphic simplicial complexes are considered to carry essentially the same combinatorial data, and we identify them accordingly.
Any simplicial complex $\Delta$ on $V(\Delta)$ is isomorphic to some simplicial complex $\Delta'$ on $\{1, \cdots, n\}$,
where $n=|V(\Delta)|$.
Thus, we may always assume that the vertex set $V(\Delta)$ of a simplicial complex $\Delta$ is the finite set of integers $\{1, \cdots, n_{\Delta}\}$ for some integer $n_{\Delta} > 0$.
\end{remark}

Let $\kk$ be a field.
Let $n \ge 1$ be an integer
and let \( S = \kk[x_1, \ldots, x_n] \) be the polynomial ring in \( n \) variables over \( \kk \), equipped with the grading \( \deg(x_i) = 1 \) for \( i = 1, \ldots, n \).  
Given a simplicial complex \( \Delta \) on $\{1, \cdots, n\}$,
recall that the \emph{Stanley--Reisner ideal} \( I_{\Delta} \subseteq S \) is an ideal of $S$ generated by all squarefree monomials \( x_{i_1}x_{i_2} \cdots x_{i_s} \) such that \( \{i_1, \ldots, i_s\} \notin \Delta \).
$\kk[\Delta]:=S/I_{\Delta}$ is called the {\it Stanley--Reisner ring} of $\Delta$. Let us remind that the {\it dimension} of $\sigma\in\Delta$ is $\dim(\sigma)=|\sigma|-1$ and the {\it dimension} of $\Delta$ is $\dim(\Delta)=\max\{\dim(\sigma):\sigma\in\Delta\}$.
We say that \( \sigma \in \Delta \) is an \emph{\( i \)-face} of \( \Delta \) if \( \dim(\sigma) = i \).  
A face of \( \Delta \) that is maximal with respect to inclusion is called a \emph{facet}.
We write \( \FF(\Delta) \) for the set of facets of \( \Delta \).
Given a face \( \sigma \in \Delta \), we define its 
\emph{link} in \( \Delta \) as
$\lk_{\Delta}(\sigma) = \{ \tau \in \Delta : \tau \cup \sigma \in \Delta \text{ and } \tau \cap \sigma = \emptyset \}.$

{Let $\Delta$ be a simplicial complex.  
Two vertices $v, w \in V(\Delta)$ are said to be {\it connected} if there exists a finite sequence $\sigma_1, \ldots, \sigma_k \in \Delta$ in $\Delta$ such that
$v \in \sigma_1$,
$w \in \sigma_k$, and
 $\sigma_i \cap \sigma_{i+1} \neq \emptyset$ for all $1 \leq i < k$.
This defines an equivalence relation $\sim$ on $V(\Delta)$.
Let $\{v_1, \ldots, v_n\}$ be a complete set of representatives of $V(\Delta)$ under the equivalence relation $\sim$.  
For each $1 \le i \le n$, define
\[
V_i := \{ v \in V(\Delta) : v \sim v_i \}, \quad
\Delta_i := \{ \sigma \in \mathcal{P}(V_i) : \sigma \in \Delta \}.
\]
Then $\Delta_i$ is a simplicial complex on the vertex set $V_i$.  
$\Delta_1, \cdots, \Delta_n$
are called \textit{connected components} of $\Delta$.
If the number of connected components is $1$, then $\Delta$ is said to be \textit{connected}.}


Finally, we introduce the fact that the Stanley--Reisner ring \( \kk[\Delta] \) arising from a disconnected simplicial complex \( \Delta \) is the fiber product over \( \kk \) of the Stanley--Reisner rings corresponding to its connected components.
We define the disjoint union of simplicial complexes as follows.

\begin{definition}\label{def:disjunionsimpcomp}
{Let \( \Delta_1 \) and \( \Delta_2 \) be simplicial complexes.  
For each \( i \in \{1, 2\} \) and each \( \sigma \in \Delta_i \), define $\tilde{\sigma}_i := \{ (v, i) : v \in \sigma \}$.  
Then $\Delta_i \times \{i\} := \{ \tilde{\sigma}_i : \sigma \in \Delta_i \}$ is a simplicial complex on $\{ (v, i) : v \in V(\Delta_i) \}$ and $\Delta_i$ is isomorphic to $\Delta_i \times \{i\}$.
Moreover, the (disjoint) union $(\Delta_1 \times \{1\}) \sqcup (\Delta_2 \times \{2\})$  
is a simplicial complex on $V(\Delta_1 \times \{1\}) \sqcup V(\Delta_2 \times \{2\})$.
We call
$(\Delta_1 \times \{1\}) \sqcup (\Delta_2 \times \{2\})$ the \textit{disjoint union} of $\Delta_1$ and $\Delta_2$, and denote it by $\Delta_1 \sqcup \Delta_2$.
For a collection of simplicial complexes $\Delta_1, \ldots, \Delta_n$,  
their disjoint union $\bigsqcup_{i=1}^n \Delta_i$ is defined inductively by  
$\bigsqcup_{i=1}^n \Delta_i := \left( \bigsqcup_{i=1}^{n-1} \Delta_i \right) \bigsqcup \Delta_n$.}
\end{definition}

\begin{remark}
{Any simplicial complex $\Delta$ is isomorphic to the disjoint union of connected simplicial complexes.  
Indeed, let $\Delta_1, \ldots, \Delta_n$ be the connected components of $\Delta$.  
Then $\Delta$ is isomorphic to the disjoint union $\bigsqcup_{i=1}^n \Delta_i$.  
Note that $\dim(\Delta) = \max \{ \dim(\Delta_i) : 1 \le i \le n \}$.}
\end{remark}

\begin{lemma}\label{lem:SRfiberproduct}
For each $i=1,\cdots,n$,
let $\Delta_i$ be a simplicial complex
and set $\Delta=\bigsqcup_{i=1}^n \Delta_i$.
Then we have $$\kk[\Delta] \cong \kk[\Delta_1] \times_\kk \kk[\Delta_2] \times_\kk \cdots \times_\kk \kk[\Delta_n].$$
\end{lemma}
\begin{proof}
It suffices to prove the case where \( n=2 \).
Let \( A = \kk[\Delta] \), \( \mathfrak{a}_1 = \MM_{\kk[\Delta_2]} A \) and \( \mathfrak{a}_2 = \MM_{\kk[\Delta_1]} A \).  
Then we have \( \mathfrak{a}_1 \cap \mathfrak{a}_2 = (0) \).  
Applying \cite[Lemma 3.1]{ogoma1984existence} with \( \mathfrak{a}_0 := \mathfrak{a}_1 + \mathfrak{a}_2 = \MM_A \)
and \( A_0 := A / \mathfrak{a}_0 = \kk \), we obtain \( A \cong A_1 \times_\kk A_2 \), where \( A_i := A / \mathfrak{a}_i \cong \kk[\Delta_i] \) for \( i=1,2 \).
Thus we have $\kk[\Delta]\cong \kk[\Delta_1] \times_\kk \kk[\Delta_2]$.
\end{proof}

\begin{remark}
Let $n \in \ZZ$.
For all \( 1 \leq i \leq n \),
let $\Delta_i$ be a connected simplicial complex and let $A_i=\kk[\Delta_i]$.
Set $\Delta=\bigsqcup_{i=1}^n \Delta_i$ and
$R=\kk[\Delta]$.
Then $R$ is Cohen--Macaulay on the punctured spectrum if and only if $A_i$ is Cohen--Macaulay on the punctured spectrum for any $1 \le i \le n$.
\end{remark}
\begin{proof}
Let $\Delta$ be a simplicial complex. It is well known that the Stanley--Reisner ring $\kk[\Delta]$ is Cohen--Macaulay on the punctured spectrum if and only if $\kk[\lk_\Delta(x_i)]$ is Cohen--Macaulay for every vertex $i \in V(\Delta)$, due to the well-known isomorphism
$\kk[\Delta]_{x_i} \cong \kk[\lk_\Delta(x_i)][x_i, x_i^{-1}]$.
The condition that $\kk[\lk_\Delta(x_i)]$ is Cohen--Macaulay for all $j \in V(\Delta)$ is in turn equivalent to requiring that $\kk[\lk_{\Delta_i}(x_j)]$ is Cohen--Macaulay for all $1 \le i \le n$ and for all $j \in V(\Delta_i)$, which is further equivalent to the assertion that each ring $A_i$ is Cohen--Macaulay on the punctured spectrum for $1 \le i \le n$.
\end{proof}

\section{Trace ideals of fiber products}\label{sec:trfib}
In this section, we study the canonical trace of a fiber product.
Let $(A, \fkm_A)$, $(B, \fkm_B)$, and $(T, \fkm_T)$ be positively graded Noetherian ring.
Let $f:A \to T$ and $g: B \to T$ be graded homomorphisms of graded rings. Set 
\begin{center}
$R = A \times_T B$ \quad and \quad $d:=\dim(R)=\max\{\dim(A), \dim(B)\}$
\end{center}
(\autoref{rem:12}). We assume that \autoref{conditiona}.
In addition, we use the following notation. 

\begin{setup}\label{setup:domainor}
Let us define graded ring homomorphisms as follows:
\begin{align*}
\pi_1 : R \to A; (a,b)\mapsto a & \quad \text{and} \quad \pi_2 : R \to B; (a,b)\mapsto b\\
\iota_1: A \to R; a\mapsto (a,f(a)) & \quad \text{and} \quad \iota_2: B \to R; b\mapsto (g(b),b).
\end{align*}
We note that $\pi_1$ and $\pi_2$ are surjective since $f$ and $g$ are surjective (see \autoref{conditiona}). 
\end{setup}

\begin{remark}\label{rem32}
Assume that $T$ is a field.
Let $I\subsetneq A$ be a graded ideal of $A$
and let $J\subsetneq B$ be a graded ideal of $B$.
Then $IR \cap JR=(0)$, that is, $IR+JR=IR \oplus JR$.
In particular, $\mm_R=\mm_A R \oplus \mm_B R$~(see \autoref{lem:27}).
\end{remark}

For a commutative ring \( S \) and an ideal \( I \subseteq S \), we denote the {\it radical} of \( I \) by \( \sqrt{I}_S \).
When there is no risk of confusion, we simply write \( \sqrt{I} \).

\begin{lemma}\label{lem:extension}
Assume that $T$ is a field. Let $I\subsetneq A$ be a graded ideal of $A$
and let $J\subsetneq A$ be a graded ideal of $B$.
Then the following hold:
\begin{itemize}
\item[\rm (1)] $I R \cap A=I$ and $J R \cap B=J$;
\item[\rm (2)] $\sqrt{I R}_R=\sqrt{I}_A R \oplus \sqrt{(0)}_B R$;
\item[\rm (3)] $\sqrt{IR\oplus JR}_R=\sqrt{I R}_R \oplus \sqrt{J R}_R=\sqrt{I}_A R \oplus \sqrt{J}_B R$.
\end{itemize}
\begin{proof}
(1): 
We only need to show that $I R \cap A=I$. Since $IR \cap A \supseteq I$ holds in general, it suffices to show that $IR \cap A \subseteq I$.
Let $x \in IR \cap A$ be arbitrary.
By definition of $x$, there exist elements $x_i \in I$, $a_i \in A$, and $b_i\in B$ for $i=1,2, \dots, n$ such that
\[
(x, f(x)) = \sum_{i=1}^n (a_i, b_i)(x_i, f(x_i)) = \left( \sum_{i=1}^n a_i x_i, \sum_{i=1}^n b_i f(x_i) \right).
\]
In particular, we have $x = \sum_{i=1}^n a_i x_i \in I$, as desired.

(2):
%
Note that
$\sqrt{IR}_R \supseteq \sqrt{I}_A R$
holds by \cite[Exercise~1.18]{atiyah2018introduction}.
Similarly,
we have
$\sqrt{IR}_R \supseteq \sqrt{(0)}_R \supseteq \sqrt{(0)}_B R$.
Therefore,
$\sqrt{IR}_R \supseteq \sqrt{I}_A R \oplus \sqrt{(0)}_B R$.
We prove the reverse inclusion.
Let $(x,y) \in \sqrt{IR}_R$.
Then there exists an integer $n>0$ such that
$(x,y)^n = (x^n, y^n) \in IR$.
Thus, $x^n \in I$ and, since $I\subseteq \fkm_A$, $y^n = 0$.
Hence $x \in \sqrt{I}_A$ and $y \in \sqrt{(0)}_B$.
Since $x \in \mathfrak{m}_A$ and $y \in \mathfrak{m}_B$,
we have $(x,0), (0,y) \in R$.
Therefore,
$(x,y) = (x,0)+(0,y) \in \sqrt{I}_A R \oplus \sqrt{(0)}_B R$,
as desired.

(3): 
By (2), we observe 
\[
\sqrt{IR}_R \oplus \sqrt{JR}_R = (\sqrt{I}_A R \oplus \sqrt{(0)}_B R) + (\sqrt{(0)}_A R \oplus \sqrt{J}_B R) =\sqrt{I}_A R \oplus \sqrt{J}_B R.
\]
Hence, it suffices to show that
$\sqrt{IR \oplus JR}_R = \sqrt{IR}_R \oplus \sqrt{JR}_R$. The inclusion $\sqrt{IR\oplus JR}_R \supseteq \sqrt{IR}_R \oplus \sqrt{JR}_R$ follows from the inclusions $\sqrt{IR\oplus JR}_R \supseteq \sqrt{IR}_R$ and $\sqrt{IR\oplus JR}_R \supseteq \sqrt{JR}_R$, which are clear. The rest is to prove that $\sqrt{IR \oplus JR}_R \subseteq \sqrt{IR}_R \oplus \sqrt{JR}_R$.
Let $(x,y) \in \sqrt{IR \oplus JR}_R$ be arbitrary.
Then there exists an integer $n>0$ such that
$(x,y)^n = (x^n, y^n) \in IR \oplus JR$.
Since $I \neq A$ and $J \neq B$, it follows that $x^n \in I$ and $y^n \in J$, and in particular
$(x,0) \in \sqrt{IR}_R$ and $(0,y) \in \sqrt{JR}_R$.
Therefore,
we have
$(x,y) = (x,0)+(0,y) \in \sqrt{IR}_R \oplus \sqrt{JR}_R$.
\end{proof}
\end{lemma}

\begin{lemma}\label{lem:AhOK}
Let $M$ be a graded $A$-module.
Consider $M$ as a graded $R$-module via $\pi_1: R \to A$.
Then the following hold:
\begin{itemize}
\item[\rm (1)]
For a homogeneous element $(a,b) \in
\tr_R(M)$, we have $b \in (0):_B \MM_B$;
\item[\rm (2)]
If
$B \neq B_0$,
then we have $\tr_R(M) \neq R$.
\end{itemize}
\begin{proof}
(1):
Since \((a, b) \in \operatorname{tr}_R(M)\),  
we can write \((a, b) = \sum_{i=1}^n \phi_i(a_i)\),  
where \(\phi_i \in {}^*\operatorname{Hom}_R(M, R)\) and \(a_i \in M\) are homogeneous elements.
Take any homogeneous element \(b' \in \mathfrak{m}_B\).  
Note that \((0, b') \in R\). Then we have
$(0, bb') = (0, b') (a, b) = \sum_{i=1}^n \phi_i((0, b) a_i) = \sum_{i=1}^n \phi_i(0) = (0, 0)$.
Thus, \(bb' = 0\), and hence \(b \in (0) :_B \mathfrak{m}_B\).

(2):
Suppose that \(\operatorname{tr}_R(M) = R\).  
Then there exists a graded \(R\)-homomorphism \(\varphi: M \to R\)  
and a homogeneous element \(x \in M\) such that \(\varphi(x)=1_R\).
Take a non-zero homogeneous element \(b'' \in B \setminus B_0\).  
Then \((0, b'') \in R\), and we compute
$(0, b'') = (0, b'') \varphi(x) = \varphi((0, b'') x) = \varphi(0x) = \varphi(0) = (0, 0)$.
It follows that \(b'' = 0\), which is a contradiction.
\end{proof}
\end{lemma}

\begin{remark}\label{rem:ann}
Assume that $T$ is a field.
Then we have $(0):_R\mm_R=((0):_A \mm_A)R\oplus ((0):_B \mm_B)R$.
\end{remark}


\begin{proposition}\label{lem:OK1}
Let $M$ be a graded $A$-module.
Consider $M$ as a graded $R$-module via $\pi_1: R \to A$. Suppose that $T$ is a field.
Then the following hold:
\begin{itemize}
\item[\rm (1)]
$\tr_R(M)
\subseteq
\tr_A(M)R\oplus ((0):_B \MM_B)R$;
\item[\rm (2)]
If $\tr_{A}(M) \subseteq {\MM_A}$,
then we have $\tr_R(M)=\tr_A(M)R\oplus ((0):_B \MM_B)R$;
\item[\rm (3)]
If $B \neq B_0$ and $\tr_{A}(M) =A$,
then we have
$\tr_R(M)=\MM_{A} R\oplus ((0):_B \MM_B)R$.
\end{itemize}
\begin{proof}
(1): 
Let \( \varphi: M \to R \) be a graded $R$-homomorphism and $a\in M$ be a homogeneous element. It suffices to show that $\varphi(a)\in \tr_A(M)R\oplus ((0):_B \MM_B)R$ (see \autoref{rem:interestingfiniteness}).
Write \( \varphi(a) = (\alpha, \beta) \), where \( \alpha \in A \) and \( \beta \in B \) are homogeneous elements. 
We note that  \( \pi_1  \varphi : M \to A\) is an \( A \)-homomorphism since \( \varphi : M \to R \) is an \( R \)-homomorphism and \( \pi_1 \) is the projection onto the first component. Thus, \( \alpha = \pi_1 \varphi(a) \in \tr_A(M) \).

Suppose that $\alpha\not\in \fkm_A$. Then, $\tr_A(M)=A$; hence, $\tr_A(M)R=AR=\iota_1(A)R=R$. Hence, the assertion holds. 

Suppose that $\alpha\in \fkm_A$. By \autoref{lem:AhOK}~(1), \( \beta \in 0 :_B \mm_B\). Therefore, we have
$\varphi(a) = (\alpha, 0) + (0, \beta) = \iota_1(\alpha) + \iota_2(\beta) \in \tr_A(M)R + ((0):_B \mm_B)R$.

(2):
By (1), it is enough to show that $\tr_R(M) \supseteq \tr_A(M)R$ and $\tr_R(M) \supseteq ((0):_B \mm_B)R$. We have 
\[
\tr_R(M) \supseteq (0):_R\mm_R\supseteq ((0):_B \mm_B)R
\]
by \autoref{lem:thanks}~(3) and \autoref{rem:ann}. 
It remains to show that $\tr_R(M) \supseteq \tr_A(M)R$. 
Since $\tr_A(M)R$ is generated by $\iota_1(\tr_A(M))$ as a graded $R$-module,  
it suffices to prove that $\Im(\iota_1\varphi) \subseteq \tr_R(M)$  
for every graded $A$-homomorphism $\varphi : M \to A$. Since $\Im(\varphi) \subseteq \tr_A(M)\subseteq \fkm_A=\Ker(f)$ by the hypothesis, one can easily check that the map 
\[
\iota_1\varphi: M\to R; a\mapsto (\varphi(a), f(\varphi(a))) = (\varphi(a), 0)
\]
is $R$-linear. It follows that $\Im(\iota_1\varphi) \subseteq \tr_R(M)$. 

(3): We first prove the case $M=A$. 
We show that $\tr_R(A) \supseteq \MM_A R + ((0):_B \MM_B)R$.  
As in (2), it follows from \autoref{lem:thanks}~(3) and \autoref{rem:ann} that $\tr_R(A) \supseteq ((0):_B \mm_B)R$.  
Therefore, it suffices to show that $\MM_A R \subseteq \tr_R(A)$.
Take any homogeneous element $a \in \MM_A$.  
We claim that $\iota_1(a) = (a, 0) \in \tr_R(A)$.  
Consider the $A$-homomorphisms $t_1: A \to A$, $x \mapsto ax$, and $t_2: A \to B$, $x \mapsto 0$.  
Note that $f t_1 = g t_2 = 0$ because $f(a) = 0$.  
By the universal property of the fiber product, we get the $R$-homomorphism $\phi: A \to R$ defined by $\phi(x) = (ax, 0)$.  
Hence, $(a, 0) = \phi(1) \in \tr_R(A)$.

Next, we show that $\tr_R(A) \subseteq \MM_A R + ((0):_B \MM_B)R$.
Take any homogeneous element $x \in \tr_R(A)$.  
Then there exist graded $R$-homomorphisms $\varphi_i: A \to R$ and homogeneous elements $a_i \in A$ such that $x = \sum_{i=1}^{n} \varphi_i(a_i)$.  
Write $\varphi_i(a_i) = (\alpha_i, \beta_i)$, where $\alpha_i \in A$ and $\beta_i \in B$ are homogeneous.
If $\alpha_i \notin \MM_A$, then $\varphi_i(a_i)$ is a unit in $R$ since $\alpha_i=f(\alpha_i)=g(\beta)=\beta$ in a field $T$. This implies $\tr_R(M) = R$, contradicting \autoref{lem:AhOK}~(2).  
Thus, $\alpha_i \in \MM_A$.  
On the other hand, we have $\beta_i \in (0):_B \MM_B$ by \autoref{lem:AhOK}~(1).  
Hence, $\varphi_i(a_i) =
\iota_1(\alpha_i) + \iota_2(\beta_i) \in \MM_A R + ((0):_B \MM_B)R$ for each $i$.  
Therefore, $x \in \MM_A R + ((0):_B \MM_B)R$,
as desired.

Now, we prove the assertion (3) in general form. Since $\tr_{A}(M) =A$, there exists $s>0$ such that $M\cong A^s \oplus M'$ as $R$-modules, where $M'$ has no $A$-free summand~(see {\cite[Remark~2.2~(4)]{celikbas2023traces}}). Thus, $\tr_A(M')\subseteq \fkm_A$. 
By (2) and the case $M=A$, it follows that 
\begin{align*}
\tr_R(M) =& \tr_R(A) + \tr_R (M') \\
=& (\MM_A R \oplus [(0):_B \MM_B]R) + (\tr_A(M')R\oplus [(0):_B \MM_B]R)\\
=& \MM_A R \oplus [(0):_B \MM_B]R.
\end{align*}
\end{proof}
\end{proposition}

\begin{lemma}\label{lem37}
Suppose that $\dim (A)=0$. Then,  $((0):_A \fkm_A)\omega_A=(0):_{\omega_A}\fkm_A$. 
\end{lemma}

\begin{proof}
We first prove that $((0):_A \fkm_A)\omega_A\ne 0$. If $((0):_A \fkm_A)\omega_A=(0)$, then we have
$$(0):_A \fkm_A \subseteq (0):_{A}\omega_A=\Ann_A(\omega_A).$$
By noting that $\dim(A)=0$ and thus $A$ is Cohen-Macaulay, $\omega_A$ is faithful, that is, $\Ann_A(\omega_A)=(0)$. It follows that $(0):_A \fkm_A=(0)$, which is a contradiction since $\dim (A)=0$. Hence, $((0):_A \fkm_A)\omega_A\ne (0)$. We observe that 
$(0)\subsetneq ((0):_A \fkm_A)\omega_A \subseteq (0):_{\omega_A}\fkm_A$.
On the other hand, we note that $(0):_{\omega_A}\fkm_A \cong {}^*\Hom_A(A/\fkm_A, \omega_A)\cong A/\fkm_A$ (\cite[Definition 3.6.8]{bruns1998cohen}).
It follows that $((0):_A \omega_A)\omega_A$ is a non-zero $A/\mm_A$-submodule of $(0):_{\omega_A} \mm_A \cong A/\mm_A$. Therefore, we obtain  $((0):_A \fkm_A)\omega_A=(0):_{\omega_A}\fkm_A$. 
\end{proof}

\begin{proposition}\label{thm:main0}
Suppose that $T$ is a field. Set $d=\dim(R)$ and 
\[
\omega_{A, B}:=\begin{cases}
    \omega_{A} \oplus \omega_{B} & \text{if } \dim(A)=\dim(B) \\
    \omega_{A} & \text{if } \dim(A)>\dim(B) \\    
    \omega_{B} & \text{if } \dim(A)<\dim(B).
\end{cases}
\]
If $d \ne 1$, then $\tr_R(\omega_R) = \tr_R(\omega_{A, B})$.
\end{proposition}
\begin{proof}
If $d\ge 2$, the assertion is clear since $\omega_R \cong \omega_{A, B}$ by \autoref{lem:important1} (3). Suppose that $d=0$. We have that $\omega_{A,B}=\omega_{A} \oplus \omega_{B}$ since $d=\dim (A)=\dim (B)=0$ (see \autoref{rem:12}). We note that $\omega_T\cong T\cong R/\fkm_R$ since $T$ is a field. Thus, we observe that 
\[0\rightarrow R/\fkm_R \xrightarrow{\phi} \omega_{A} \oplus \omega_{B} \rightarrow \omega_R
\rightarrow 0
\]
by \autoref{lem:important1} (1). Let $N:=\phi(R/\fkm_R)$. By the definition of $N$, $N\subseteq (0):_{\omega_A\oplus \omega_B}\fkm_R$. To apply \autoref{tracelem}, we prove that $(0):_{\omega_A\oplus \omega_B}\fkm_R = ((0):_R \fkm_R)(\omega_{A} \oplus \omega_{B})$. Indeed, let $(x,y) \in \omega_A\oplus \omega_B$. Then, noting that $\mm_R=\mm_A R \oplus \mm_B R$ by \autoref{rem32}, 
\begin{align*}
(x,y)\in (0):_{\omega_A\oplus \omega_B}\fkm_R \iff & \fkm_R(x,y)=0 \iff \fkm_A x=0 \text{\quad and \quad} \fkm_B y=0 \\
\iff & x\in (0):_{\omega_A}\fkm_A \text{\quad and \quad} y\in (0):_{\omega_B}\fkm_B.
\end{align*}
By \autoref{lem37}, the above conditions are equivalent to saying that $x\in ((0):_A \fkm_A)\omega_A$ and $y\in ((0):_B \fkm_B)\omega_B$. Therefore, by \autoref{rem:ann}, we get that 
\begin{align*}
(0):_{\omega_A\oplus \omega_B}\fkm_R =& ((0):_A \fkm_A)\omega_A  \oplus ((0):_B \fkm_B)\omega_B \\
=& [((0):_A \fkm_A)R \oplus ((0):_B \fkm_B)R](\omega_A\oplus \omega_B)\\
=& ((0):_R\mm_R)(\omega_A\oplus \omega_B). 
\end{align*}
Hence, we get that $N\subseteq (0):_{\omega_A\oplus \omega_B}\fkm_R = ((0):_R \fkm_R)(\omega_{A} \oplus \omega_{B})$. Hence, applying \autoref{tracelem} with $M=\omega_{A} \oplus \omega_{B}$, we have 
\[
\tr_R(\omega_{A} \oplus \omega_{B}) = \tr_R([\omega_{A} \oplus \omega_{B}]/N) = \tr_R(\omega_R). 
\]
\end{proof}

\autoref{thm:main0} fails to hold when $\dim (R)=1$, see \autoref{rem:OK...}~(2).

For the sake of simplifying the statement of the main result, we introduce a convenient notation.

\begin{definition}\label{def:dagertrace}
For a Noetherian positively graded ring $(S,\MM_S)$
and a finitely generated graded $S$-module $M$,
we set
\[
\tr_S^\dagger(M) :=
\begin{cases}
\tr_S(M) & \text{if\;} \tr_S(M) \subsetneq S, \\
\mm_S & \text{if \;} \tr_S(M)=S.
\end{cases}
\]
\end{definition}

\begin{remark}\label{def:daggertukaimasita}
For a Noetherian positively graded ring $(S,\MM_S)$ possessing the graded canonical module $\omega_S$,
we have
\[
\tr_S^\dagger(\omega_S) =
\begin{cases}
\tr_S(\omega_S) & \text{if } S \text{ is not quasi-Gorenstein}, \\
\mm_S & \text{if } S \text{ is quasi-Gorenstein}
\end{cases}
\]
by \autoref{def:dagertrace} and \autoref{rem:gradedaoyamagoto}~(4).
Note that
$\tr_S^\dagger(\omega_S)=\mm_S$ if and only if
$\tr_S(\omega_S) \supseteq \mm_S$.
\end{remark}


With the above notaion, one of the main results in this paper is stated as follows.

\begin{theorem}\label{aaa}
Assume that 
$T$ is a field.
Moreover, assume
$A \neq A_0$ and
$B \neq B_0$.
If $\dim(R)\ne 1$,
then 
\[
\tr_R(\omega_R) = \begin{cases}
    \tr_A^\dagger(\omega_{A})R \oplus \tr_B^\dagger(\omega_{B})R & \text{if } \dim(A)=\dim(B) \\
    \tr_A^\dagger(\omega_{A})R \oplus ((0):_B \MM_B)R & \text{if } \dim(A)>\dim(B) \\    
    ((0):_A \MM_A)R \oplus \tr_B^\dagger(\omega_{B})R & \text{if } \dim(A)<\dim(B).
\end{cases}
\]
\end{theorem}
\begin{proof}
If $\dim(A)=\dim(B)$, by \autoref{thm:main0}, 
\[
\tr_R(\omega_R) = \tr_R(\omega_{A} \oplus \omega_{B}) = \tr_R(\omega_{A}) + \tr_R(\omega_{B}),
\]
where the second equality follows by \cite[Proposition 2.8~(ii)]{lindo2017trace}. By \autoref{lem:OK1} (2) and (3), we get $\tr_R(\omega_{A}) + \tr_R(\omega_{B}) = \tr_A^\dagger(\omega_{A})R \oplus \tr_B^\dagger(\omega_{B})R$. The other cases can be proved in a similar way.
\end{proof}

\begin{example}\label{rem:OK...}
Note that \autoref{aaa} implies that $R$ is not quasi-Gorenstein (see \autoref{thm:poyopoyo1}~(1)). On the other hand, $R$ may become quasi-Gorenstein when one of assumptions of \autoref{aaa} is removed. Hence, the assumptions of \autoref{aaa} are essential. 
\begin{itemize}
\item[\rm (1)] Suppose that $A = B = T$ and $T$ is Gorenstein. Then, $R = T \times_T T \cong T$. Hence, $\tr_R(\omega_R) = R$. Therefore, \autoref{aaa} does not hold in this case.
\item[\rm (2)] Suppose that $A \cong \kk[x]$ and $B \cong \kk[y]$. Then, $\dim (R)=\dim (A)=\dim (B)=1$. Then, $R \cong \kk[x,y]/(xy)$, and $R$ is Gorenstein; hence, $\tr_R(\omega_R) = R$. Therefore, \autoref{aaa} does not hold in this case. This example also shows that \autoref{thm:main0}, which plays a crucial role in the proof of \autoref{aaa}, already fails to hold if $\dim (R)=1$. Indeed, in this case, $\omega_{A,B}:=\omega_A \oplus \omega_B\cong A\oplus B$ and thus 
\[
\tr_R(\omega_{A,B}) = \tr_R(A\oplus B) =\tr_R(A) \oplus \tr_R(B) =\fkm_A R \oplus \fkm_B R =\fkm_R
\]
by \cite[Proposition 2.8~(ii)]{lindo2017trace} and \autoref{lem:OK1}~(3).
Hence, we obtain that $\tr_R(\omega_R) \neq \tr_R(\omega_{A, B})$.
\end{itemize}
\end{example}

\begin{fact}\label{rem:OK......}
Although \autoref{aaa} assumes $\dim(R) \neq 1$, partial results are known for the case that $R$ is local and $\dim(R) = 1$.
When both $A$ and $B$ are one-dimensional Cohen--Macaulay generically Gorenstein local rings, and neither $A$ nor $B$ is a discrete valuation ring, a result analogous to \autoref{aaa} is established in the proof of \cite[Theorem 5.2]{endo2021almost}.
\end{fact}

In view of \autoref{rem:OK...}, we propose the following conjecture.

\begin{conjecture}\label{conjX}
When \( \dim(R) = 1 \) and neither \( A \) nor \( B \) is regular of dimension 1, \autoref{aaa} holds.
\end{conjecture}

In the remainder of this section,  
we apply \autoref{aaa} to characterize situations where the canonical trace of the fiber product $R$ contains $\mathfrak{m}_R$,  as well as those where it is an $\mathfrak{m}_R$-primary ideal.

\begin{lemma}\label{lem:minimalmultiplicity}
Assume  $T$ is a field.
Suppose that
$\dim(A)>\dim(B)$
(resp. $\dim(A)<\dim(B)$).
If $\tr_R(\omega_R) \supseteq \MM_R$,
then
$\mm_B^2=(0)$
(resp. $\mm_A^2=(0)$).
\begin{proof}
Consider the case where
$\dim(A) > \dim(B)$.
If $B$ is a field,  
then $\mathfrak{m}_B = (0)$ and hence $\mathfrak{m}_B^2 = (0)$.  
Therefore, we may assume that $B$ is not a field.
Suppose that $\tr_R(\omega_R) \supseteq \MM_R$. Then, by \autoref{aaa}, we have
$\tr_R(\omega_R) = \tr_A^\dagger(\omega_A)R \oplus ((0) :_B \MM_B)R \supseteq \MM_A R \oplus \MM_B R$,
and thus
$((0) :_B \MM_B)R \supseteq \MM_B R$.
By \autoref{lem:extension}~(1), we deduce that
$(0) :_B \MM_B \supseteq \MM_B$.
Since $B$ is not a field, we conclude that $(0) :_B \MM_B=\MM_B$, and hence $\MM_B^2 = (0)$. The case where $\dim(A) < \dim(B)$ can be treated similarly.
\end{proof}
\end{lemma}

\begin{theorem}\label{thm:poyopoyo1}
Assume \( T \) is a field. 
Furthermore, suppose that \( A \ne A_0 \) and \( B \ne B_0 \). 
If \( d=\dim(R) \ne 1 \), the following statements hold:
\begin{itemize}
\item[\rm (1)] We have $\tr_R(\omega_R) \neq R$,
that is,
$R$ never becomes quasi-Gorenstein;
\item[\rm (2)]
The following conditions are equivalent;
\begin{enumerate}
  \item[\rm (a)] \(\operatorname{tr}_R(\omega_R) = \mathfrak{m}_R \).
\item[\rm (b)] One of the following conditions holds:
\begin{itemize}
\item[\rm (i)]
\( \dim(A) = \dim(B) \),
\( \operatorname{tr}_A(\omega_A) \supseteq \mathfrak{m}_A \) and \( \operatorname{tr}_B(\omega_B) \supseteq \mathfrak{m}_B \);
\item[\rm (ii)] \( \dim(A) > \dim(B)=0 \),
    \( \operatorname{tr}_A(\omega_A) \supseteq \mathfrak{m}_A \) and \( \mathfrak{m}_B^2 = 0 \);
    \item[\rm (iii)] \( \dim(B) > \dim(A)=0 \),
    \( \operatorname{tr}_B(\omega_B) \supseteq \mathfrak{m}_B \) and \( \mathfrak{m}_A^2 = 0 \).
\end{itemize}
\end{enumerate}
\item[\rm (3)]
The following conditions are equivalent;
\begin{enumerate}
\item[\rm (a)] \(\operatorname{tr}_R(\omega_R)\) is \( \mathfrak{m}_R\)-primary.
\item[\rm (b)] One of the following conditions holds:
\begin{itemize}
\item[\rm (i)]
\( \dim(A) = \dim(B) \),
\( \sqrt{\operatorname{tr}_A(\omega_A)}_A \supseteq \mathfrak{m}_A \) and \( \sqrt{\operatorname{tr}_B(\omega_B)}_B \supseteq \mathfrak{m}_B \);
\item[\rm (ii)] \( \dim(A) > \dim(B)=0 \), \( \sqrt{\operatorname{tr}_A(\omega_A)}_A \supseteq \mathfrak{m}_A \);
    \item[\rm (iii)] \( \dim(B) > \dim(A)=0 \),
    \( \sqrt{\operatorname{tr}_B(\omega_B)}_B \supseteq \mathfrak{m}_B \).
\end{itemize}
\end{enumerate}
\end{itemize}
\end{theorem}
\begin{proof}
(1):
Noting that $\tr_A^{\dagger}(\omega_A) \subseteq \mm_A$
and
$\tr_B^{\dagger}(\omega_B) \subseteq \mm_B$,
this follows
by \autoref{aaa}.

(2):
Note that $\mm_A^2=(0)$ if and only if $(0):_A\mm_A \supseteq \mm_A$.
$(b) \Rightarrow (a)$ follows from \autoref{aaa}.

We show $(a) \Rightarrow (b)$.
When $\dim(A) = \dim(B)$, it follows from \autoref{aaa} that $\tr_A(\omega_A) \supseteq \mm_A$ and $\tr_B(\omega_B) \supseteq \mm_B$.
When $\dim(A) > \dim(B)$, we have $\tr_A(\omega_A) \supseteq \mm_A$ by \autoref{aaa}.
Furthermore, we have $\mm_B^2 = (0)$ by \autoref{lem:minimalmultiplicity}.
Similarly, when $\dim(A) < \dim(B)$, it follows from \autoref{aaa} and \autoref{lem:minimalmultiplicity} that $\tr_B(\omega_B) \supseteq \mm_B$ and $\mm_A^2 = (0)$.

(3):
Note that \( \dim(A) = 0 \) (resp. $\dim(B)=0$) is equivalent to \( (0) :_A \mathfrak{m}_A \supseteq \mathfrak{m}_A^k \) (resp. \( (0) :_B \mathfrak{m}_B \supseteq \mathfrak{m}_B^k \)) for some \( k \geq 0 \). \((b) \Rightarrow (a)\) follows from \autoref{aaa}.

We prove \((a) \Rightarrow (b)\).
Suppose that \(\operatorname{tr}_R(\omega_R)\) is \(\mathfrak{m}_R\)-primary.
When \(\dim(A) = \dim(B)\), by \autoref{aaa}, we have
$\operatorname{tr}_R(\omega_R) = \operatorname{tr}_A^\dagger(\omega_A)R \oplus \operatorname{tr}_B^\dagger(\omega_B)R.$
In this case, by \autoref{lem:extension}~(3) and the assumption, we obtain
$$\sqrt{\operatorname{tr}_A^\dagger(\omega_A)}_AR \oplus \sqrt{\operatorname{tr}_B^\dagger(\omega_B)}_BR = \sqrt{\operatorname{tr}_R(\omega_R)}_R = \mathfrak{m}_R = \mathfrak{m}_A R \oplus \mathfrak{m}_B R,$$
and thus \(\sqrt{\operatorname{tr}_A^\dagger(\omega_A)}_AR = \mathfrak{m}_A R\) and \(\sqrt{\operatorname{tr}_B^\dagger(\omega_B)}_BR = \mathfrak{m}_B R\).
It follows that
$\sqrt{\operatorname{tr}_A^\dagger(\omega_A)}_A = \mathfrak{m}_A$ and
$\sqrt{\operatorname{tr}_B^\dagger(\omega_B)}_B = \mathfrak{m}_B$ by \autoref{lem:extension}~(1).
Thus, we conclude that \(\sqrt{\operatorname{tr}_A(\omega_A)}_A \supseteq \mathfrak{m}_A\) and \(\sqrt{\operatorname{tr}_B(\omega_B)}_B \supseteq \mathfrak{m}_B\).
When \(\dim(A) > \dim(B)\), by \autoref{aaa}, we have
$\operatorname{tr}_R(\omega_R) = \operatorname{tr}_A^\dagger(\omega_A)R \oplus ((0) :_B \mathfrak{m}_B)R$.
By an argument similar to the above and using \autoref{lem:extension}~(1) and (3), we obtain
$\sqrt{\operatorname{tr}_A(\omega_A)}_A \supseteq \mm_A$
and
$\sqrt{(0) :_B \mathfrak{m}_B}_B \supseteq \mm_B$.
The latter condition is equivalent to \((0) :_B \mm_B \supseteq \mm_B^k\) for some \( k \geq 0 \), that is, \(\dim(B) = 0\).
The case \(\dim(A) < \dim(B)\) can be shown in a similar manner.
\end{proof}

\begin{corollary}
If $T$ is a field, $\dim(R) \neq 1$ and $\tr_R(\omega_R) \supseteq \MM_R$,
then
$\tr_{A}(\omega_{A}) \supseteq \MM_A$
and
$\tr_{B}(\omega_{B}) \supseteq \MM_B$.
\end{corollary}
\begin{proof}
Note that if $\mm_A^2 = (0)$, then we have $\tr_R(\omega_A) \subseteq (0) :_A \mm_A = \mm_A$ by \autoref{lem:thanks}~(3).  
Thus the assertion follows from \autoref{thm:poyopoyo1}~(2).
\end{proof}

\begin{remark}\label{rem:TooRoughExplanation?}
Assume that \( T \) is a field.
Further assume that \( A \neq A_0 \) and \( B \neq B_0 \).
Assume that \( d = \dim(A) = \dim(B) = 1 \), and that \( A \) and \( B \) are Cohen--Macaulay, generically Gorenstein, positively graded rings, not both of which are regular of dimension 1. Then, a statement analogous to \autoref{thm:poyopoyo1} holds.
That is, we have $\operatorname{tr}_R(\omega_R) = \operatorname{tr}_A^{\dagger}(\omega_A)R \oplus \operatorname{tr}_B^{\dagger}(\omega_B)R$.
\begin{proof}
We first show that
$R_{\mathfrak{m}_R} \cong A_{\mathfrak{m}_A} \times_T B_{\mathfrak{m}_B}$,
where the fiber product on the right-hand side is taken over the common residue field $T$ with respect to natural surjections
$f: A_{\mathfrak{m}_A} \rightarrow A_{\mathfrak{m}_A}/\mm_A A_{\mm_A}=T$ and $g: B_{\mathfrak{m}_B} \rightarrow B_{\mathfrak{m}_B}/\mm_B B_{\mm_B}=T$.
By the universal properties of fiber product and localization,
we can define a ring homomorphism
$\varphi : R_{\mathfrak{m}_R} \rightarrow A_{\mathfrak{m}_A} \times_T B_{\mathfrak{m}_B}$
by
\[
\varphi\left( \frac{(a,b)}{(s,t)} \right) = \left( \frac{a}{s}, \frac{b}{t} \right),
\]
for \((a,b) \in R = A \times_T B\), and \((s,t) \in R \setminus \mathfrak{m}_R\) (hence, $s=t\ne 0$ in T).
Then, by the definition of $\varphi$, it is straightforward to verify that \( \varphi \) is an isomorphism. Hence, 
$R_{\mathfrak{m}_R} \cong A_{\mathfrak{m}_A} \times_T B_{\mathfrak{m}_B}$.

Since \( R_{\mathfrak{m}_R} \cong A_{\mathfrak{m}_A} \times_T B_{\mathfrak{m}_B} \),
it follows from \cite[Proposition 2.8~(viii)]{lindo2017trace}
and the proof of \cite[Theorem 5.2]{endo2021almost} that
\[
\operatorname{tr}_R(\omega_R) R_{\mathfrak{m}_R}
\cong
\operatorname{tr}_{R_{\mathfrak{m}_R}}(\omega_{R_{\mathfrak{m}_R}})
=
\operatorname{tr}_{A_{\mathfrak{m}_A}}^{\dagger}(\omega_{A_{\mathfrak{m}_A}}) R_{\mathfrak{m}_R}
\oplus
\operatorname{tr}_{B_{\mathfrak{m}_B}}^{\dagger}(\omega_{B_{\mathfrak{m}_B}}) R_{\mathfrak{m}_R}.
\]

Moreover, by \autoref{lem:OK1}~(2) and (3) for the local case, we have
\[
\operatorname{tr}_{A_{\mathfrak{m}_A}}^{\dagger}(\omega_{A_{\mathfrak{m}_A}}) R_{\mathfrak{m}_R}
\oplus
\operatorname{tr}_{B_{\mathfrak{m}_B}}^{\dagger}(\omega_{B_{\mathfrak{m}_B}}) R_{\mathfrak{m}_R}
=
\operatorname{tr}_{R_{\mathfrak{m}_R}}(\omega_{A_{\mathfrak{m}_A}} \oplus \omega_{B_{\mathfrak{m}_B}}).
\]

Since
\(
\omega_{A_{\mathfrak{m}_A}} \cong \omega_A \otimes_R R_{\mathfrak{m}_R}
\)
and
\(
\omega_{B_{\mathfrak{m}_B}} \cong \omega_B \otimes_R R_{\mathfrak{m}_R}
\)
as $R_{\fkm_R}$-modules, we obtain
\[
\omega_{A_{\mathfrak{m}_A}} \oplus \omega_{B_{\mathfrak{m}_B}}
\cong
(\omega_A \oplus \omega_B) \otimes_R R_{\mathfrak{m}_R}.
\]

Consequently, by \cite[Proposition 2.8~(viii)]{lindo2017trace}, we have
\[
\operatorname{tr}_{R_{\mathfrak{m}_R}}(\omega_{A_{\mathfrak{m}_A}} \oplus \omega_{B_{\mathfrak{m}_B}})
=
\operatorname{tr}_{R_{\mathfrak{m}_R}}((\omega_A \oplus \omega_B) \otimes_R R_{\mathfrak{m}_R})
\cong
\operatorname{tr}_R(\omega_A \oplus \omega_B) R_{\mathfrak{m}_R}.
\]

Therefore, we obtain
$\operatorname{tr}_R(\omega_R) R_{\mathfrak{m}_R}
\cong
\operatorname{tr}_R(\omega_A \oplus \omega_B) R_{\mathfrak{m}_R}$.
By \cite[Corollary~3.9]{hashimoto2023indecomposability},
we have
$\operatorname{tr}_R(\omega_R) \cong \operatorname{tr}_R(\omega_A \oplus \omega_B)$
as graded \( R \)-modules.
From this, it follows that
$\operatorname{tr}_R(\omega_R) = \operatorname{tr}_R(\omega_A \oplus \omega_B)$.
Since
\(
\operatorname{tr}_R(\omega_A) = \operatorname{tr}_A^{\dagger}(\omega_A) R
\)
and
\(
\operatorname{tr}_R(\omega_B) = \operatorname{tr}_B^{\dagger}(\omega_B) R
\)
by \autoref{lem:OK1}~(2) and (3), we obtain
\[
\operatorname{tr}_R(\omega_R)
=
\operatorname{tr}_A^{\dagger}(\omega_A) R
\oplus
\operatorname{tr}_B^{\dagger}(\omega_B) R.
\]
\end{proof}
\end{remark}

In particular, the following result holds in the case where the fiber product is reduced and of dimension greater than one.

\begin{corollary}\label{cor:poyopoyo}
Assume that \( T \) is a field. 
Furthermore, suppose that \( A \ne A_0 \) and \( B \ne B_0 \).
Then, if
$R$ is reduced and $\dim(R) \neq 1$, the following hold:
\begin{itemize}
\item[\rm (1)] $\tr_R(\omega_R)=\MM_R$
if and only if
$\dim(A)=\dim(B)$,
$\tr_{A}(\omega_{A}) \supseteq \MM_{A}$
and
$\tr_{B}(\omega_{B}) \supseteq \MM_{B}$.
\item[\rm (2)]
$\sqrt{\tr_R(\omega_R)}_R=\MM_R$
if and only if
$\dim(A)=\dim(B)$,
$\sqrt{\tr_{A}(\omega_{A})}_A \supseteq \MM_{A}$
and
$\sqrt{\tr_{B}(\omega_{B})}_B \supseteq \MM_{B}$.
\end{itemize}
\begin{proof}
Since \( R \) is reduced,
\( A \) and \( B \) are both reduced by \autoref{lem:extension}~(1) and (2).
If $\dim(A) = 0$, then $A$ must be a field because $A$ is reduced.  
As $A$ is positively graded, we have $A = A_0$. This yields a contradiction.  
Therefore, $\dim(A) \ge 1$.  
Similarly, we obtain $\dim(B) \ge 1$
and hence 
the claim follows from \autoref{thm:poyopoyo1}.
\end{proof}
\end{corollary}

Finally, we describe the canonical trace of the fiber product of multiple rings.

\begin{theorem}\label{thm:poyopoyo}
Let $n \geq 2$ be an integer, and assume that $T$ is a field.  
For each $1 \leq i \leq n$, let $(A_i, \mathfrak{m}_{A_i})$ be a Noetherian positively graded ring such that the set of its degree-zero homogeneous elements, $[A_i]_0$, is equal to $T$,
and let $f_i: A_i \rightarrow A_i/(A_i)_{>0}
\cong T$
be the canonical graded surjection.
Let $R = A_1 \times_T \cdots \times_T A_n$ be the fiber product of the $A_i$ over $T$ with respect to $f_1, \dots, f_n$.
Set
$X = \{ 1 \leq i \leq n \mid \dim(A_i) = \dim(R) \}$.
Assume $\dim(R) \neq 1$.
Then the following assertions hold:
\begin{itemize}
\item[\rm (1)]
The following hold:
\begin{itemize}
\item[\rm (i)] {We have $\tr_R(\omega_{R})=\bigoplus_{i \in X} \tr_{A_i}^{\dagger}(\omega_{A_i})R$;}
\item[\rm (ii)] We have $\tr_R(\omega_R) \neq R$,
that is,
$R$ never becomes quasi-Gorenstein;
\item[\rm (iii)]
\( \tr_R(\omega_R) = \mathfrak{m}_R \)
if and only if \( \tr_{A_i}(\omega_{A_i}) \supseteq \mathfrak{m}_{A_i} \) for all \( i \in X \) and \( \mm_{A_i}^2 = 0 \)
for all
$i \notin X$.
\item[\rm (iv)]
\( \sqrt{\operatorname{tr}_R(\omega_R)}_R = \mathfrak{m}_R \) if and only if \( \sqrt{\operatorname{tr}_{A_i}(\omega_{A_i})}_{A_i} \supseteq \mathfrak{m}_{A_i} \) for all \( i \in X \) and \( \dim(A_i) = 0 \) for all \( i \notin X \);
\end{itemize}
\item[\rm (2)]
Assume $R$ is reduced.
Then the following hold:
\begin{itemize}
\item[\rm (i)] $\sqrt{\tr_R(\omega_R)}_R \supseteq \MM_R$
if and only if
$\sqrt{\tr_{A_i}(\omega_{A_i})}_{A_i} \supseteq \mm_{A_i}$ 
and
$\dim(A_i)=\dim(R)$
for all $i$.
\item[\rm (ii)]
$\tr_R(\omega_R) \supseteq \MM_R$
if and only if
$\tr_{A_i}(\omega_{A_i}) \supseteq \mm_{A_i}$
and
$\dim(A_i)=\dim(R)$
for all $i$.
\end{itemize}
\end{itemize}
\end{theorem}
\begin{proof}
By induction on $n$, the assertion can be reduced to the case $n=2$.
(1) follows from \autoref{aaa} and \autoref{thm:poyopoyo1}.
(2) follows from \autoref{cor:poyopoyo}.
\end{proof}

The following is a computational formula for the canonical traces of Stanley--Reisner rings that arise from disconnected simplicial complexes.

\begin{theorem}\label{MainTHM:B}
Let $n \ge 2$ be an integer.
For $i=1,\ldots,n$,
let $\Delta_i$ be a simplicial complex and let $A_i=\kk[\Delta_i]$.
Set $\Delta=\bigsqcup_{i=1}^n \Delta_i$ and $R=\kk[\Delta]$,
{and assume that $\Delta$ is not the discrete simplicial complex on two vertices.}
Then the following hold:
\begin{itemize}
\item[\rm (1)]
We have
$$\tr_{\kk[\Delta]}(\omega_{\kk[\Delta]})
= \bigoplus_{\substack{1 \le i \le n \\ \dim(\Delta_i) = \dim(\Delta)}}
\tr_{\kk[\Delta_i]}^{\dagger}(\omega_{\kk[\Delta_i]})\kk[\Delta];$$
\item[\rm (2)]
$\sqrt{\tr_R(\omega_{R})}_R=\MM_{R}$
if and only if
$\sqrt{\tr_{A_i}(\omega_{A_i})}_{A_i} \supseteq \MM_{A_i}$
and
$\dim(\Delta_i)=\dim(\Delta)$
for all $1 \le i \le n$;
\item[\rm (3)]
$\tr_R(\omega_{R})=\MM_{R}$
if and only if
$\tr_{A_i}(\omega_{A_i}) \supseteq \MM_{A_i}$
and
$\dim(\Delta_i)=\dim(\Delta)$
for all $1 \le i \le n$.
\end{itemize}
\end{theorem}
\begin{proof}
This follows from \autoref{lem:SRfiberproduct} and \autoref{thm:poyopoyo}.
\end{proof}

\section{Noetherian rings of Teter type and their applications}\label{sec:teterStanley}
In this section, we apply \autoref{MainTHM:B} to characterize Stanley--Reisner rings whose canonical traces contain the graded maximal ideal, in terms of simplicial complexes.
As a preparation for this, we first generalize \cite[Theorem 2.1]{gasanova2022rings} to Noetherian rings that are not necessarily Cohen--Macaulay.

\subsection{Noetherian rings of Teter type}
The following is a natural generalization of the notion of rings of Teter type, which was defined for Cohen--Macaulay rings in \cite{gasanova2022rings}, to Noetherian (positively graded) rings.

\begin{definition}[cf. {\cite{gasanova2022rings}}]
Let $R$ be a Noetherian positively graded ring.
We say that $R$ is of {\it Teter type}
if there exists a graded $R$-homomorphism
$\phi:\omega_R \rightarrow R$ such that
$\tr_R(\omega_R)=\phi(\omega_R)$.
\end{definition}

\begin{remark}
In \cite{gasanova2022rings}, the aforementioned concept was introduced for non-Gorenstein (Cohen--Macaulay) rings.
Note that our definition includes those that are quasi-Gorenstein.
\end{remark}

Recall that $R$ is called unmixed if $\dim(R)=\dim(R/\pp)$ for any $\pp \in \Ass(R)$.

\begin{theorem}\label{Thm:nonCMTeter}
Let $R$ be a unmixed generically Gorenstein graded ring.
Then $R$ is of Teter type if and only if $R$ is quasi-Gorenstein.
\end{theorem}
\begin{proof}
If $R$ is quasi-Gorenstein,
the graded isomorphism $\omega_R \cong R$ is the defining map of Teter type.
Assume that $R$ is of Teter type.
If $\dim(R)=0$, then $R$ is Gorenstein because $R$ is generically Gorenstein.
Thus we may assume that $\dim(R)>0$.
Assume that \( R \) is not quasi-Gorenstein.  
Since the trace is compatible with localization
(see \cite[Proposition 2.8~(viii)]{lindo2017trace}), we have  
${}^*\Supp(\tr_R(\omega_R)) \subseteq {}^*\Supp(\omega_R)$.
Since \( R \) is not quasi-Gorenstein, it follows from \autoref{rem:gradedaoyamagoto}~(4) that \( \tr_R(\omega_R) \neq R \).  
Moreover, since 
$\height(\tr_R(\omega_R))
> 0$~(see \autoref{rem:Canon.NZD}),
we can choose \( \pp \in \Min(R/\tr_R(\omega_R)) \) such that \( \height(\pp) > 0 \).  
Thus, we have  
$\pp \in \Min(R/\tr_R(\omega_R)) \cap {}^*\Supp(\omega_R)$.
Since $\omega_R$ as well as $\tr_R(\omega_R)$ localize at $\pp \in {}^*\Supp(\omega_R)$,
and by
\autoref{rem:gradedaoyamagoto}~(5),
we may replace $R$ by $R_\pp$, and hence we may assume that $R$ is a generically Gorenstein ring of Teter type but not quasi-Gorenstein with $\dim(R)>0$ and that $\tr_R(\omega_R)$ is $\MM_R$-primary.
Choose an epimorphism $\phi : \omega_R\to \tr_R(\omega_R)$.
Here, we show that \( C = \ker(\phi) \) is zero.  

Indeed, from the short exact sequence  
$0 \rightarrow C \rightarrow \omega_R \rightarrow \tr_R(\omega_R) \rightarrow 0$
and
$\Ass_R(\omega_R) = \Assh(R)$ by \autoref{rem:gradedaoyamagoto}~(2),
we obtain the inclusion  
$\Ass_R(C) \subseteq \Ass_R(\omega_R) = \Assh(R)$.
Moreover,
we have $\Assh(R) \subseteq {}^*\Spec(R) \setminus {}^*\Supp(C)$
because $R$ is generically Gorenstein.
Hence, $\Ass(C) \subseteq {}^*\Spec(R) \setminus {}^*\Supp(C)$.
Since $\Ass_R(C) \subseteq {}^*\Supp(C)$,
we obtain $\Ass_R(C)=\emptyset$.
This yields a contradiction.

Therefore
$\omega_R \cong \tr_R(\omega_R)$.
If \( \dim(R) \geq 2 \), notice that \( \omega_R \) satisfies \( (S_2) \) and \( R/\tr_R(\omega_R) \) is an Artinian ring, we see that \( \depth(R) = 0 \) by applying the local cohomology functor $H_{\mm_R}^{*}(-)$ to the
exact sequence  
$0 \rightarrow \omega_R \rightarrow R \rightarrow R/\tr_R(\omega_R) \rightarrow 0$.
Then we have \( \MM_R \in \Ass(R) \), so that \( R \cong R_{\MM_R} \) is quasi-Gorenstein, leading to a contradiction.  
Thus we may assume \( \dim(R) = 1 \).  
Since \( \depth(R) > 0 \), it follows that \( R \) is Cohen–Macaulay.
Therefore, the assertion follows from \cite[Theorem 2.1]{gasanova2022rings}.
\end{proof}

\begin{remark}
The outline of the proof of \autoref{Thm:nonCMTeter} follows \cite[Theorem 2.1]{gasanova2022rings}, but it has been skillfully generalized in a technical manner so that it can also be applied to non-Cohen--Macaulay cases.
\end{remark}

\subsection{Applications to Stanley--Reisner rings}\label{stanleyNICE}
In this subsection, we apply the previous results to Stanley--Reisner rings.
We define some fundamental terms.

We say that \( \Delta \) is \emph{pure} if \( \dim(\sigma) = \dim(\Delta) \) for every facet \( \sigma \in \FF(\Delta) \).
A simplicial complex \( \Delta \) is said to be \emph{strongly connected} if, for any \( \sigma, \tau \in \FF(\Delta) \), there exist facets  
$\sigma_0, \sigma_1, \sigma_2, \ldots, \sigma_k \in \FF(\Delta)$
such that \( \sigma_0 = \sigma \), \( \sigma_k = \tau \), and  
$\dim(\sigma_i \cap \sigma_{i-1}) = \dim(\Delta) - 1$ for $i = 1, \ldots, k$.

\begin{remark}
Any strongly connected simplicial complex is pure.
\end{remark}

A face $\tau\in\Delta$ is a {\it cone} if $\tau\subseteq \sigma$ for any $\sigma\in\FF(\Delta)$.
A simplicial complex $\Delta$ is called {\it normal} if $\lk_{\Delta}(\sigma)$ is connected for any $i$-face with $i\leq \dim(\Delta)-2$.

\begin{remark}\label{rem:EX}
Since $\dim(\emptyset)=-1$ and $\lk_{\Delta}(\emptyset) =\Delta$, a normal simplicial complex of dimension at least 1 is connected. It turns out that, if $\Delta$ is normal, then $\lk_{\Delta}(\sigma)$
is strongly connected for all $\sigma\in \Delta$~(see \cite[Proposition 11.7]{Bj}). In particular a normal simplicial complex is pure. 
Indeed, it is known that $\Delta$ is normal if and only if $\kk[\Delta]$ satisfies Serre's condition $(S_2)$.
\end{remark}

\begin{remark}\label{rem:sr-S2-punctured}
Let \(\Delta\) be a connected simplicial complex.
If \(\kk[\Delta]\) satisfies \((S_2)\) on the punctured spectrum,
then $\Delta$ is normal.
In particular,
If
$\kk[\Delta]$
is
quasi-Gorenstein on the punctured spectrum,
then $\Delta$ is normal.
\end{remark}
\begin{proof}
Let \(\sigma\in\Delta\) with \(\dim\sigma\leq \dim\Delta-2\). If \(\sigma=\emptyset\), then \(\lk_\Delta(\sigma)=\Delta\), which is connected by assumption.

For a nonempty face \(\sigma\), set
$\mathfrak p_\sigma=(x_i\mid i\notin\sigma)\kk[\Delta]$.
Then \(\mathfrak p_\sigma\neq\mathfrak m_{\kk[\Delta]}\). By the standard localization formula for Stanley--Reisner rings, we have
\[
\kk[\Delta]_{\mathfrak p_\sigma}\cong \kk(x_i\mid i\in\sigma)[\lk_\Delta(\sigma)]_{\mathfrak n},
\]
where \(\mathfrak n\) denotes the irrelevant maximal ideal of \(\kk(x_i\mid i\in\sigma)[\lk_\Delta(\sigma)]\). Hence, if \(\kk[\Delta]\) satisfies \((S_2)\) on the punctured spectrum, then $\kk[\lk_\Delta(\sigma)]$ satisfies \((S_2)\). In particular, when \(\dim\lk_\Delta(\sigma)\geq 1\), the complex \(\lk_\Delta(\sigma)\) is connected.
Thus \(\Delta\) is normal.

In particular, if \(\kk[\Delta]\) is quasi-Gorenstein on the punctured spectrum, then \(\kk[\Delta]\) satisfies \((S_2)\) on the punctured spectrum by \cite[(1.10)]{aoyama1983some}. Hence \(\Delta\) is normal.
\end{proof}

For any commutative ring $R$,
${H}_i(\Delta, R)$ (resp. $\widetilde{H}_i(\Delta, R)$) denotes the $i$-th (resp. reduced) homology group of $\Delta$ with coefficients in $R$.
For a given $\sigma \in \Delta$,
we say that
$(*)$ holds for $\sigma$ if the following condition is satisfied:
\[
\widetilde{H}_{i}{(\lk_\Delta(\sigma), \kk)} = 
\begin{cases}
    \kk, & \text{if } i = \dim (\lk_\Delta(\sigma)), \\
    0, & \text{otherwise}.
\end{cases}
\]
Notice that $\Delta$ is connected if $(*)$ holds for $\emptyset \in \Delta$.

\begin{definition}
Let $\Delta$ be a $d$-dimensional simplicial complex and let $\kk$ be a field. We say that $\Delta$ is:
\begin{itemize}
\item
a {\it $\kk$-homology sphere} if
$(*)$ holds for any $\sigma \in \Delta$.
\item a {\it $\kk$-homology manifold}
if $\Delta$ is connected and $(*)$ holds for any $\sigma \in \Delta$ such that $\sigma \neq \emptyset$.
\item a {\it pseudomanifold} if it is strongly connected and any $(d-1)$-face is contained in exactly two facets (which means that $\sigma$ holds $(*)$ for every facet $\sigma \in \FF(\Delta)$).
\end{itemize}
Suppose that $\Delta$ is a pseudomanifold.
We say that
\( \Delta \) is \emph{orientable} if \( H_d(\Delta; \ZZ) \neq 0 \), and \emph{\( \kk \)-orientable} if \( H_d(\Delta; \kk) \neq 0 \).
\end{definition}

\begin{remark}
Let $\Delta$ be a simplicial complex and let $\kk$ be a field.
By a classical theorem of Hochster,
$\kk[\Delta]$ is Gorenstein if and only if $\lk_\Delta(\sigma)$ is a $\kk$-homology sphere, where $\sigma$ is the cone face of $\Delta$ maximal by inclusion (e.g. see \cite[Theorem 5.6.1]{bruns1998cohen}).
Note that \( \Delta \) is a \( \kk \)-homology manifold if and only if \( \lk_\Delta(v) \) is a homology sphere for every vertex \( v \in V(\Delta) \).
Moreover, the link of a pseudomanifold is not necessarily a pseudomanifold itself; for example, it may fail to be strongly connected (as in the case of a pinched torus).  
For this reason, it is sometimes more natural to work with normal pseudomanifolds than pseudomanifolds.
\end{remark}

\begin{remark}\label{lem:1}
Let $\Delta$ be a simplicial complex and set $R=\kk[\Delta]$.
If $R$ is quasi-Gorenstein on the punctured spectrum,
then $\lk_\Delta(\sigma)$ is quasi-Gorenstein for any $\emptyset \neq \sigma \in \Delta$.
\begin{proof}
Since $R$ is quasi-Gorenstein on the punctured spectrum,
$R$ satisfies $(S_2)$ on the punctured spectrum.
Thus $R$ satisfies $(S_2)$ by \autoref{rem:sr-S2-punctured}
and so $\Delta$ is pure.
Since \(\Delta\) is pure, we have \({}^*\!\Supp(\omega_R) = {}^*\!\Spec(R)\) by \autoref{rem:gradedaoyamagoto}~(2).
We proceed by induction on \( n = \dim(\sigma) \).
First, consider the case \( n = 1 \), that is, when \( \sigma \) is a vertex \( i \in V(\Delta) \).
Let
$$\mathfrak{p}_{x_i} := (x_j : j \neq i)R \in \Spec(R)$$
be the homogeneous prime ideal of \( R \) corresponding to \( i \).
The set of all homogeneous elements in \( R \setminus \mathfrak{p}_{x_i} \) coincides with the multiplicative subset
$S_{x_i} := \{ x_i^k : k \in \mathbb{N} \} \subset R$.
Since \( R \) is quasi-Gorenstein on the punctured spectrum, it follows that \( R_{x_i} \cong R_{(\pp_{x_i})} \) is a quasi-Gorenstein graded ring by \autoref{rem:gradedaoyamagoto}~(1).
Moreover, there is an isomorphism
$R_{x_i} \cong \kk[\operatorname{lk}_\Delta(x_i)][x_i, x_i^{-1}]$
(see \cite[p.~62]{stanley2007combinatorics}).
Hence, since \( \kk[\operatorname{lk}_\Delta(x_i)][x_i, x_i^{-1}] \) is quasi-Gorenstein, it follows from \autoref{rem:QGPolyExtension} that \( \kk[\operatorname{lk}_\Delta(x_i)] \) is also quasi-Gorenstein.

Now, assume that the claim holds for \( n \), and consider the case \( n+1 \).
In this situation, we can write \( \sigma = \{v\} \cup \tau \),
and define \( \Delta' := \operatorname{lk}_\Delta(\tau) \).
Then we have
$\operatorname{lk}_\Delta(\sigma) = \operatorname{lk}_{\Delta'}(\{v\})$
(see the proof of \cite[Proposition~6.3.15]{villarreal2001monomial}).
Therefore, it suffices to show that \( \operatorname{lk}_{\Delta'}(\{v\}) \) is quasi-Gorenstein.

Since \( \dim(\tau) \leq n \), by the induction hypothesis, \( \Delta' \) is quasi-Gorenstein.
In particular, the Stanley--Reisner ring \( \kk[\Delta'] \) is quasi-Gorenstein (on the punctured spectrum).
Applying the base case for \( n = 1 \), we conclude that \( \operatorname{lk}_{\Delta'}(\{v\}) \) is quasi-Gorenstein.
\end{proof}
\end{remark}

\begin{remark}\label{Rem:useful?}
Let $\Delta$ be a simplicial complex and set $R=\kk[\Delta]$.
Then the following hold:
\begin{itemize}
\item[\rm (1)]
If $R$ is Cohen--Macaulay on the punctured spectrum and $\Delta$ is connected,
then $\Delta$ is normal;
\item[\rm (2)]
If $\Delta$ is pure, then
$\tr_R(\omega_R)$ describes non-quasi-Gorenstein locus of $R$.
\end{itemize}
\begin{proof}
(1):
Take any $i$-face $\sigma \in \Delta$ with $i \le \dim(\Delta)-2$.
When \( i = -1 \),
since $\dim(\emptyset)=-1$ and \( \lk_\Delta(\emptyset) = \Delta \) is connected, the statement holds.
When \( i \geq 0 \), by assumption, \( \lk_\Delta(\sigma) \) is Cohen--Macaulay, which in particular implies that it is connected, so the statement holds.

(2): Noting that \( \Delta \) is pure if and only if \( R \) is unmixed,
the claim follows from \autoref{rem:gradedaoyamagoto}~(2), (5).
\end{proof}
\end{remark}

The following
is a generalization of \cite[Proposition 3.4]{miyashita2024canonical} to the non-Cohen–Macaulay case.
The proof is almost the same as that of \cite[Proposition 3.4]{miyashita2024canonical}, but it requires \autoref{Thm:nonCMTeter}.

\begin{proposition}\label{p:nonCMnearlyG=G}
Let $\Delta$ be a non-orientable normal pseudomanifold,
and $R=\kk[\Delta]$.
Then
$\tr_{R}(\omega_R)\subseteq \MM_{R}^2$.
\end{proposition}
\begin{proof}
Suppose that $\tr_R(\omega_R)\nsubseteq \MM_R^2$.
In this case, following the proof of \cite[Proposition 3.4]{miyashita2024canonical},  
there exists \( \phi \in \Hom_R^*(R,\omega_R) \) such that \( \tr_R(\omega_R) = \phi(\omega_R) \).  
This contradicts \autoref{Thm:nonCMTeter}.
\end{proof}

We establish a slight generalization of \cite[Proposition 3.4]{miyashita2024canonical}.

\begin{proposition}\label{prop:main1}
Let $\Delta$ be a strongly connected simplicial complex with $\dim(\Delta) \ge 2$ and set $R=\kk[\Delta]$.
If there is no cone point of $\Delta$ and $\lk_\Delta(\sigma)$ is quasi-Gorenstein for any $\sigma \in \Delta$ with $\dim(\sigma) \ge \dim(\Delta)-2$,
then $\Delta$ is a pseudomanifold.
Moteover,
if $\Delta$ is normal,
then we have either $R$ is quasi-Gorenstein or $\tr_R(\omega_{R}) \subseteq \MM_{R}^2$.
\end{proposition}
\begin{proof}
For any $\sigma \in \Delta$ with $\dim(\sigma) \ge \dim(\Delta) - 2$,  
note that $\lk_\Delta(\sigma)$ is quasi-Gorenstein if and only if it is Gorenstein because $\dim(\lk_\Delta(\sigma)) \le 1$.
Under this assumption, we can check that \( \Delta \) is a pseudomanifold, as in the proof of \cite[{Proposition~2.9}]{miyashita2024canonical}.
When \( \Delta \) is a normal pseudomanifold and $R$ is not quasi-Gorenstein,
we have
$\Delta$ is non-orientable by \cite[Theorem 4.4]{varbaro2024lefschetz}.
Then we obtain $\tr_R(\omega_R) \subseteq \MM_{R}^2$ 
by \autoref{p:nonCMnearlyG=G}.
\end{proof}

\begin{remark}\label{rem:noconepoint}
Let $\Delta$ be a connected simplicial complex 
and set $R=\kk[\Delta]$.
If $R$ is not quasi-Gorenstein but quasi-Gorenstein on the punctured spectrum,
then there is no cone point of $\Delta$.
\end{remark}
\begin{proof}
Assume that there exists a cone point $i$ of $\Delta$.
Note that $R \cong \kk[\lk_\Delta(x_i)][x_i]$ because $i$ is a cone point of $\Delta$.
Then $R_{x_i} \cong \kk[\lk_\Delta(x_i)][x_i,x_i^{-1}]$ is quasi-Gorenstein by \autoref{lem:1},
so $R \cong \kk[\lk_\Delta(x_i)][x_i]$ is quasi-Gorenstein by \autoref{rem:QGPolyExtension}.
\end{proof}

\begin{corollary}
Let $\Delta$ be a connected simplicial complex with $\dim(\Delta) \ge 2$ and set $R=\kk[\Delta]$.
If $\tr_R(\omega_R) \supseteq \MM_R^2$ and $R$ is not quasi-Gorenstein,
then we have $\tr_R(\omega_R)=\MM_R^2$.
\end{corollary}
\begin{proof}
$R$ is quasi-Gorenstein on the punctured spectrum by \autoref{rem:gradedaoyamagoto}~(6)
so that there is no cone point of $\Delta$ by \autoref{rem:noconepoint}.
Moreover, since $R$ is quasi-Gorenstein on the punctured spectrum,
we have $\lk_\Delta(\sigma)$ is quasi-Gorenstein for any $\sigma \in \Delta$ with $\dim(\sigma) \ge \dim(\Delta)-2$ by \autoref{lem:1}.
Thus we have $\tr_R(\omega_R)=\MM_R^2$ by \autoref{prop:main1}.
\end{proof}

\begin{theorem}\label{thm:good1}
Let $\Delta$ be a connected simplicial complex with $\dim(\Delta) \ge 2$ and let $R=\kk[\Delta]$. Then the following are equivalent:
\begin{itemize}
\item[\rm (1)] $\tr_R(\omega_{R})=R$, that is, $R$ is quasi-Gorenstein;
\item[\rm (2)] $\tr_R(\omega_{R}) \supseteq \MM_{R}$~(e.g., $R$ is nearly Gorenstein);
\item[\rm (3)]
$R$ is quasi-Gorenstein on the punctured spectrum and
$[\tr_R(\omega_{R})]_1 \neq (0)$;
\item[\rm (4)]
$\Delta$ is strongly connected,
$\kk[\lk_\Delta(\sigma)]$ is quasi-Gorenstein for any $\sigma \in \Delta$ such that $\dim(\sigma) \ge \dim(\Delta)-2$
and
$\kk[\lk_\Delta(x_i)]$ is quasi-Gorenstein for any cone point of $\Delta$~(e.g., there is no cone point of $\Delta$)
and $[\tr_R(\omega_{R})]_1 \neq (0)$.
\end{itemize}
\end{theorem}
\begin{proof}
$(1) \Rightarrow (2)$ is clear.
$(2) \Rightarrow (3)$ follows from \autoref{rem:gradedaoyamagoto}~(6).
We prove \((3)\Rightarrow(4)\). By \autoref{lem:1}, it is enough to show that \(\Delta\) is strongly connected. Since \(R\) is quasi-Gorenstein on the punctured spectrum, \autoref{rem:sr-S2-punctured} implies that \(\Delta\) is normal. Hence \(\Delta\) is strongly connected, because \(\Delta\) is normal and connected; see \cite[Proposition~11.7]{Bj}.
We now show the implication \( (4) \Rightarrow (1) \).  
Note that under this assumption, \(\kk[\lk_\Delta(\sigma)]\) is Gorenstein for any \( \sigma \in \Delta \) such that \( \dim(\sigma) \ge \dim(\Delta) - 2 \) because $\dim(\lk_\Delta(\sigma)) \le 1$.
Furthermore, we may assume that \( \Delta \) has no cone point by \autoref{rem:noconepoint}.
In this case, by \autoref{prop:main1}, we have either \( R \) is quasi-Gorenstein or \( \tr_R(\omega_R) \subseteq \MM_R^2 \).  
However, the condition \( \tr_R(\omega_R) \subseteq \MM_R^2 \) contradicts to \( [\tr_R(\omega_R)]_1 \neq 0 \).  
Thus, we conclude that \( R \) is quasi-Gorenstein.
\end{proof}

\begin{remark}
\autoref{thm:good1} is a generalization of \cite[Corollary 3.5]{miyashita2024canonical}.
\autoref{thm:good1}~(4) does not even assume that 
$R$ is quasi-Gorenstein on the punctured spectrum.
\end{remark}

\begin{corollary}[{cf. \cite[Corollary 3.5]{miyashita2024canonical}}]
Let $\Delta$ be a connected simplicial complex with $\dim(\Delta) \ge 2$.
If $\tr_{\kk[\Delta]}(\omega_{\kk[\Delta]})
\supseteq \mm_{\kk[\Delta]}$,
then $\kk[\Delta]$ is quasi-Gorenstein.
\end{corollary}
\begin{proof}
It follows from (3) $\Rightarrow$ (1) of \autoref{thm:good1}.
\end{proof}

The following result is a generalization of \cite[Theorem A]{miyashita2024canonical}.

\begin{theorem}[\autoref{thm:canonicalsquareee}]\label{thm:canonicalsquare}
Let $\Delta$ be a connected simplicial complex and set $R=\kk[\Delta]$.
Then the following hold:
\begin{itemize}
\item[\rm (1)] 
If $\sqrt{\tr_R(\omega_R)} \supseteq \mm_R$, then $\tr_R(\omega_R) \supseteq \mm_R$ or
$[\tr_R(\omega_R)]_1 =(0)$.
In the latter case, $\Delta$ is a non-orientable pseudomanifold;
\item[\rm (2)]
Assume that $R$ is Cohen--Macaulay on the punctured spectrum.
Then
$\sqrt{\tr_R(\omega_R)} \supseteq \mm_R$
if and only if
$\tr_R(\omega_R) = \MM_R^i$ for some $i \in \{0,1,2\}$;
\item[\rm (3)]
The following are equivalent:
\begin{itemize}
\item[\rm (a)] $\tr_R(\omega_{R})=\MM_R$;
\item[\rm (b)] $R$ is Cohen--Macaulay and $\tr_R(\omega_R)=\MM_R$ (i.e., $R$ is non-Gorenstein and nearly Gorenstein in the sense of \cite{herzog2019trace});
\item[\rm (c)] $\Delta$ is isomorphic to a path of length $n \geq 3$.
\end{itemize}
\item[\rm (4)]
$\tr_R(\omega_R)=\mm_R^2$ and $R$ is Cohen--Macaulay on the punctured spectrum if and only if $\Delta$ is a non-orientable $\kk$-homology manifold.
\begin{itemize}
\item[\rm (a)] $R$ is Cohen--Macaulay on the punctured spectrum and $\tr_R(\omega_R)=\mm_R^2$;
\item[\rm (b)] $R$ is Gorenstein on the punctured spectrum and $\tr_R(\omega_R)=\mm_R^2$;
\item[\rm (c)] $\Delta$ is a non-orientable $\kk$-homology manifold.
\end{itemize}
\end{itemize}
\end{theorem}
\begin{proof}
(1):
Assume that  $R$ is quasi-Gorenstein on the punctured spectrum.
Note that $\Delta$ is
strongly connected
by \cite[Proposition 11.7]{Bj}.
If $\dim(R) \le 1$, then $R$ is Cohen--Macaulay because $\Delta$ is connected. Thus the claim follows from \cite[Theorem~A (X)]{miyashita2024canonical}.
If $\dim(\Delta) \geq 2$, then by \autoref{prop:main1} and \autoref{rem:noconepoint}, either $R$ is quasi-Gorenstein or $\Delta$ is a pseudomanifold.~
In the case where $R$ is quasi-Gorenstein, we obtain $\tr_R(\omega_R) = R$.  
Otherwise,
$\Delta$ is a normal pseudomanifold.
According to \cite[Theorem 4.4]{varbaro2024lefschetz},
$R$ is quasi-Gorenstein if $\Delta$ is $\kk$-orientable.
Thus we have either $\tr_R(\omega_R)=R$ or $\tr_R(\omega_R) \subseteq \MM_R^2$ by
\autoref{p:nonCMnearlyG=G}.
Therefore, in each case, it follows that
$\tr_R(\omega_R) \subseteq \mm_R$,
or $[\tr_R(\omega_R)]_1 =(0)$
and $\Delta$ is a non-orientable pseudomanifold.

(2):
Assume $\sqrt{\tr_R(\omega_R)} \supseteq \mm_R$.
Then
\( R \) is Gorenstein on the punctured spectrum by \autoref{rem:gradedaoyamagoto}~(6).
If \( \tr_R(\omega_R) \not\supseteq \mm_R \), then \( \Delta \) is a non-orientable normal pseudomanifold and \( \tr_R(\omega_R) \subseteq \mm_R^2 \) by (1) and \autoref{p:nonCMnearlyG=G}.  
Thus \( \Delta \) is a \( \kk \)-homology manifold by by \cite[Lemma~2.7]{miyashita2024canonical}.
Then we have \( \tr_R(\omega_R) = \mm_R^2 \) by \cite[Proposition~3.10]{miyashita2024canonical}.
Therefore, \( \tr_R(\omega_R) = \mm_R^i \) for some \( i \in \{0,1,2\} \).
The converse is clear.

(3):
$(b) \Leftrightarrow (c)$ is known from \cite[Theorem A~(Y)]{miyashita2024canonical}.  
It suffices to prove that $(a) \Rightarrow (b)$.  
Note that
$\Delta$ is normal by \autoref{rem:gradedaoyamagoto}~(6) and \autoref{rem:sr-S2-punctured}.
Thus $\Delta$ is pure by \cite[Proposition 11.7]{Bj}.
Assume that \( \dim(\Delta) \geq 2 \).
Then \( \tr_R(\omega_R) = R \) by \autoref{thm:good1}, which leads to a contradiction.
Thus we have \( \dim(\Delta) \le 1 \).
Then \( R \) is Cohen--Macaulay because \( \Delta \) is connected.

(4):
\((a) \Rightarrow (b)\) follows from (2).  
We now prove \((b) \Rightarrow (c)\).  
First, suppose that \(\dim(\Delta) \le 1\).  
Since \(\Delta\) is connected, \(R\) is Cohen--Macaulay.  
In this case, by \cite[Theorem 4.3~(b)]{miyashita2024levelness}, \(R\) is nearly Gorenstein.  
However, since \(\operatorname{tr}_R(\omega_R) = \mathfrak{m}_R^2 \ne \mathfrak{m}_R\), this leads to a contradiction.  
Therefore, we may assume \(\dim(\Delta) \ge 2\).  
Now, since \(\operatorname{tr}_R(\omega_R) = \mathfrak{m}_R^2 \ne R\), \(R\) is not quasi-Gorenstein by \autoref{rem:gradedaoyamagoto}~(4).
Then $\Delta$ has no cone point by \autoref{rem:noconepoint}.
Hence, \(\Delta\) is a pseudomanifold by \autoref{prop:main1}.
It follows from \cite[Lemma 2.7]{miyashita2024canonical} that $\Delta$ is a $\kk$-homology manifold.
Note that $\Delta$ is normal. Moreover, according to \cite[Theorem 4.4]{varbaro2024lefschetz} and \autoref{rem:gradedaoyamagoto}~(4),
$\tr_R(\omega_R)=R$ if $\Delta$ is $\kk$-orientable.
Therefore $\Delta$ is non $\kk$-orientable because $\tr_R(\omega_R) \neq R$.
Lastly, we show $(c)\Rightarrow(a)$.
Since $\Delta$ is a $\kk$-homology manifold,
then $R$ is Gorenstein on the punctured spectrum by definition.
Thus it is Cohen--Macaulay on the punctured spectrum.
Notice that $\Delta$ is a normal pseudomanifold
because it is a $\kk$-homology manifold.
Then we have $\tr_R(\omega_R)=\mm_R^2$ by \autoref{p:nonCMnearlyG=G} and \cite[Proposition 3.10]{miyashita2024canonical}.
\end{proof}


\begin{corollary}\label{thm:good3}
 Let $\Delta$ be a connected simplicial complex and let $R=\kk[\Delta]$.
If $\tr_R(\omega_R)=\MM_R$,
then $\dim(\Delta) \le 1$.
In particular, $R$ is Cohen--Macaulay.
\begin{proof}
This follows from \autoref{thm:canonicalsquare}~(3).
\end{proof}
\end{corollary}

The following is the second main result of this paper.

\begin{theorem}\label{thm:interesting???}
Fix $2 \le n \in \ZZ$.
For all \( 1 \leq i \leq n \),
let $\Delta_i$ be a connected simplicial complex and let $A_i=\kk[\Delta_i]$.
Set $\Delta=\bigsqcup_{i=1}^n \Delta_i$ and
$R=\kk[\Delta]$, and assume that $\Delta$ is not the discrete simplicial complex on two vertices.
Then the following hold:
\begin{itemize}
\item[\rm (1)] Suppose that $R$ is Cohen--Macaulay on the punctured spectrum.
Then the following are equivalent:
\begin{itemize}
\item[\rm (a)] $\tr_R(\omega_R)$ is $\MM_R$-primary;
\item[\rm (b)] $\tr_{A_i}(\omega_{A_i}) \in \{A_i, \mm_{A_i}, \mm_{A_i}^2\}$ and \( \dim(\Delta_i) = \dim(\Delta) \) for any $i=1,\cdots,n$;
\item[\rm (c)] $\tr_R(\omega_R) \supseteq \MM_R^2$.
\end{itemize}
\item[\rm (2)] The following are equivalent:
\begin{itemize}
\item[\rm (a)] $\tr_R(\omega_R)=\MM_R$;
\item[\rm (b)] $\tr_{A_i}(\omega_{A_i}) \supseteq \MM_{A_i}$ and $\dim(\Delta_i)=\dim(\Delta)$ for
any $1 \le i \le n$;
\item[\rm (c)]
The following hold;
\begin{itemize}
\item[\rm (i)] $\dim(\Delta_i)=\dim(\Delta)$ for any $1 \le i \le n$,
\item[\rm (ii)] \( A_i \) is quasi-Gorenstein or \( \Delta_i \) is isomorphic to a path for any $1 \le i \le n$.
\end{itemize}
\end{itemize}
\item[\rm (3)] 
Suppose that $R$ is Cohen--Macaulay on the punctured spectrum.
Then the following are equivalent:
\begin{itemize}
\item[\rm (a)] $\tr_R(\omega_R)=\MM_R^2$;
\item[\rm (b)] $\tr_{A_i}(\omega_{A_i})=\MM_{A_i}^2$ and \( \dim(\Delta_i) = \dim(\Delta) \) for any \( 1 \leq i \leq n \);
\item[\rm (c)] $\Delta_i$ is a $\kk$-non-orientable $\kk$-homology manifold and \( \dim(\Delta_i) = \dim(\Delta) \) for any \( 1 \leq i \leq n \).
\end{itemize}
\end{itemize}
\end{theorem}
\begin{proof}
(1):
$(a) \Rightarrow (b)$:
By \autoref{MainTHM:B}~(2),
we have $\sqrt{\tr_{A_i}(\omega_{A_i})}_{A_i} \supseteq \MM_{A_i}$ and \( \dim(\Delta_i) = \dim(\Delta) \) for any $i = 1, \ldots, n$.  
Thus we have $\tr_{A_i}(\omega_{A_i}) \in \{ A_i, \MM_{A_i}, \MM_{A_i}^2 \}$ by \autoref{thm:canonicalsquare}~(2).
$(b) \Rightarrow (c)$:
Notice that we have $\tr_{A_i}^{\dagger}(\omega_{A_i}) \supseteq \MM_{A_i}^2$ for any $i = 1, \ldots, n$.
Then we have
$\tr_R(\omega_R) = \bigoplus_{i=1}^n \tr_{A_i}^{\dagger}(\omega_{A_i})R \supseteq \bigoplus_{i=1}^n \MM_{A_i}^2R = \MM_R^2$ by \autoref{MainTHM:B}~(1).
$(c) \Rightarrow (a)$:
Either $\tr_R(\omega_R) = R$ or $\sqrt{\tr_R(\omega_R)}_R = \mm_R$ holds from the assumption.
Since $\tr_R(\omega_R) =
\mm_R$ by \autoref{MainTHM:B}~(1),
we have $\sqrt{\tr_R(\omega_R)}_R = \mm_R$.


(2):
$(a) \Leftrightarrow (b)$ follows by \autoref{MainTHM:B}~(3).
\((b) \Leftrightarrow (c)\):
We may assume that $\dim(\Delta_i)=\dim(\Delta)$ for any $1\le i \le n$.
If \( \dim(\Delta) \ge 2 \), then
$\tr_{A_i}(\omega_{A_i})$ if and only if $A_i$ is quasi-Gorenstein by \autoref{thm:good3}.
If \( \dim(\Delta) \le 1 \), the assertion follows from \autoref{thm:canonicalsquare}~(3).

(3):
$(a) \Rightarrow (b)$:
By $(a) \Rightarrow (b)$ of (1),  
we have $\tr_{A_i}(\omega_{A_i}) \in \{ A_i, \mm_{A_i}, \mm_{A_i}^2 \}$ and \( \dim(\Delta_i) = \dim(\Delta) \) for any $i = 1, \ldots, n$.
Suppose, for contradiction, that there exists some $1 \le i_0 \le n$ such that  
$\tr_{A_{i_0}}(\omega_{A_{i_0}}) \in \{ A_{i_0}, \mm_{A_{i_0}} \}$.
Then we have
$\tr_{A_{i_0}}^{\dagger}(\omega_{A_{i_0}}) = \MM_{A_{i_0}}$
by definition.
On the other hand, since $\tr_R(\omega_R) = \MM_R^2$, we have
$\tr_{A_{i_0}}^{\dagger}(\omega_{A_{i_0}}) =  \MM_{A_{i_0}}^2$ by \autoref{lem:extension}~(1) and \autoref{MainTHM:B}~(1),
which is a contradiction.
Therefore, we obtain $\tr_{A_i}(\omega_{A_i}) = \mm_{A_i}^2$ for any $i = 1, \ldots, n$.
$(b)\Rightarrow (a)$:
By \autoref{MainTHM:B},
we have $\tr_R(\omega_R)=\bigoplus_{i=1}^n \tr_{R}(\omega_{A_i})=\bigoplus_{i=1}^n \mm_{A_i}^2 R=\MM_R^2$.
$(b) \Leftrightarrow (c)$:
This follows from \autoref{thm:canonicalsquare}~(4).
\end{proof}


\section*{Acknowledgments}
The first author was supported by JSPS KAKENHI Grant Number 24K16909.
The second author was supported by JST SPRING, Japan Grant Number JPMJSP2138.


\begin{bibdiv}
\begin{biblist}*{labels={shortalphabetic}}

\bib{aoyama1983some}{article}{
      author={Aoyama, Yoichi},
       title={Some basic results on canonical modules},
        date={1983},
     journal={Journal of Mathematics of Kyoto University},
      volume={23},
      number={1},
       pages={85\ndash 94},
}

\bib{aoyama1985endomorphism}{article}{
      author={Aoyama, Yoichi},
      author={Goto, Shiro},
      author={Others},
       title={On the endomorphism ring of the canonical module},
        date={1985},
     journal={Journal of Mathematics of Kyoto University},
      volume={25},
      number={1},
       pages={21\ndash 30},
}

\bib{atiyah2018introduction}{book}{
      author={Atiyah, Michael~F},
      author={Macdonald, Ian~Grant},
       title={Introduction to commutative algebra},
   publisher={CRC Press},
        date={2018},
}

\bib{bagherpoor2023trace}{article}{
      author={Bagherpoor, Mohammad},
      author={Taherizadeh, Abdoljavad},
       title={Trace ideals of semidualizing modules and two generalizations of nearly Gorenstein rings},
        date={2023},
     journal={Communications in Algebra},
      volume={51},
      number={2},
       pages={446\ndash 463},
}

\bib{Bj}{book}{
      author={Bj\"orner, A},
       title={Topological methods, in: Handbook of combinatorics},
}

\bib{bruns1998cohen}{book}{
      author={Bruns, W.},
      author={Herzog, J.},
       title={Cohen--Macaulay rings},
   publisher={Cambridge university press},
        date={1998},
      number={39},
}

\bib{caminata2021nearly}{article}{
      author={Caminata, Alessio},
      author={Strazzanti, Francesco},
       title={Nearly Gorenstein cyclic quotient singularities},
        date={2021},
     journal={Beitr{\"a}ge zur Algebra und Geometrie/Contributions to Algebra and Geometry},
      volume={62},
      number={4},
       pages={857\ndash 870},
}

\bib{celikbas2023traces}{article}{
      author={Celikbas, Ela},
      author={Herzog, J{\"U}rgen},
      author={Kumashiro, Shinya},
       title={Traces of semi-invariants},
        date={2023},
     journal={arXiv preprint arXiv:2312.00983},
}

\bib{dao2020trace}{article}{
      author={Dao, Hailong},
      author={Kobayashi, Toshinori},
      author={Takahashi, Ryo},
       title={Trace ideals of canonical modules, annihilators of Ext modules, and classes of rings close to being Gorenstein},
        date={2021},
     journal={Journal of Pure and Applied Algebra},
      volume={225},
      number={9},
       pages={106655},
}

\bib{endo2021almost}{article}{
      author={Endo, Naoki},
      author={Goto, Shiro},
      author={Isobe, Ryotaro},
       title={Almost Gorenstein rings arising from fiber products},
        date={2021},
     journal={Canadian Mathematical Bulletin},
      volume={64},
      number={2},
       pages={383\ndash 400},
}

\bib{ficarra2024canonical}{article}{
      author={Ficarra, Antonino},
       title={The canonical trace of Cohen--Macaulay algebras of codimension 2},
        date={2025},
     journal={Proceedings of the American Mathematical Society},
      volume={153},
      number={8},
       pages={3275\ndash 3289},
         doi={10.1090/proc/17250},
}

\bib{ficarra2024canonical!}{article}{
      author={Ficarra, Antonino},
      author={Herzog, J{\"U}rgen},
      author={Stamate, Dumitru~I},
      author={Trivedi, Vijaylaxmi},
       title={The canonical trace of determinantal rings},
        date={2024},
     journal={Archiv der Mathematik},
      volume={123},
      number={5},
       pages={487\ndash 497},
}

\bib{gasanova2022rings}{article}{
      author={Gasanova, Oleksandra},
      author={Herzog, Juergen},
      author={Hibi, Takayuki},
      author={Moradi, Somayeh},
       title={Rings of Teter type},
        date={2022},
     journal={Nagoya Mathematical Journal},
      volume={248},
       pages={1005\ndash 1033},
}

\bib{goto1978graded}{article}{
      author={Goto, Shiro},
      author={Watanabe, Keiichi},
       title={On graded rings, I},
        date={1978},
     journal={Journal of the Mathematical Society of Japan},
      volume={30},
      number={2},
       pages={179\ndash 213},
}

\bib{hall2023nearly}{article}{
      author={Hall, Thomas},
      author={K{\"o}lbl, Max},
      author={Matsushita, Koji},
      author={Miyashita, Sora},
       title={Nearly Gorenstein polytopes},
        date={2023},
     journal={Electronic Journal of Combinatorics},
}

\bib{hashimoto2023indecomposability}{article}{
      author={Hashimoto, Mitsuyasu},
      author={Yang, Yuntian},
       title={Indecomposability of graded modules over a graded ring},
        date={2023},
     journal={Kyoto Journal of Mathematics},
         note={to appear},
}

\bib{herzog2019trace}{article}{
      author={Herzog, J.},
      author={Hibi, T.},
      author={Stamate, D.~I.},
       title={The trace of the canonical module},
        date={2019},
     journal={Israel Journal of Mathematics},
      volume={233},
       pages={133\ndash 165},
}

\bib{herzog1971canonical}{article}{
      author={Herzog, J.},
      author={Kunz, E.},
       title={Die werthalbgruppe eines lokalen rings der dimension 1},
        date={1971},
     journal={Berichte der Heidelberger Akademie der Wissenschaften},
}

\bib{jafari2024nearly}{article}{
      author={Jafari, Raheleh},
      author={Strazzanti, Francesco},
      author={Armengou, Santiago~Zarzuela},
       title={On nearly Gorenstein affine semigroups},
        date={2026},
     journal={Journal of Algebra},
      volume={694},
       pages={676\ndash 702},
}

\bib{kimura2025trace}{article}{
      author={Kimura, Kaito},
       title={Trace ideals, conductors, and ideals of finite (phantom) projective dimension},
        date={2025},
     journal={arXiv preprint arXiv:2501.03442},
}

\bib{kumashiro2023trace}{article}{
      author={Kumashiro, Shinya},
       title={When are trace ideals finite?},
        date={2023},
     journal={Mediterranean Journal of Mathematics},
      volume={20},
      number={5},
       pages={278},
}

\bib{kumashiro2025nearly}{inproceedings}{
      author={Kumashiro, Shinya},
      author={Matsuoka, Naoyuki},
      author={Nakashima, Taiga},
       title={Nearly Gorenstein local rings defined by maximal minors of a 2$\times$ n matrix},
organization={Springer},
        date={2025},
   booktitle={Semigroup forum},
       pages={1\ndash 27},
}

\bib{lindo2017trace}{article}{
      author={Lindo, Haydee},
       title={Trace ideals and centers of endomorphism rings of modules over commutative rings},
        date={2017},
     journal={Journal of Algebra},
      volume={482},
       pages={102\ndash 130},
}

\bib{lu2024chain}{article}{
      author={Lu, Dancheng},
       title={The chain algebra of a pure poset},
        date={2025},
     journal={Journal of Algebraic Combinatorics},
      volume={62},
         doi={10.1007/s10801-025-01437-z},
}

\bib{lyle2024annihilators}{article}{
      author={Lyle, Justin},
      author={Maitra, Sarasij},
       title={Annihilators of (co) homology and their influence on the trace ideal},
        date={2024},
     journal={arXiv preprint arXiv:2409.04686},
}

\bib{matsumura1989commutative}{book}{
      author={Matsumura, Hideyuki},
       title={Commutative ring theory},
   publisher={Cambridge university press},
        date={1989},
      number={8},
}

\bib{miyashita2024levelness}{article}{
      author={Miyashita, Sora},
       title={Levelness versus nearly Gorensteinness of homogeneous rings},
        date={2024},
     journal={Journal of Pure and Applied Algebra},
      volume={228},
      number={4},
       pages={107553},
}

\bib{miyashita2024linear}{article}{
      author={Miyashita, Sora},
       title={A linear variant of the nearly Gorenstein property},
        date={2024},
     journal={arXiv preprint arXiv:2407.05629},
}

\bib{miyashita2025pseudo}{article}{
      author={Miyashita, Sora},
       title={When do pseudo-Gorenstein rings become Gorenstein?},
        date={2026},
     journal={Bulletin of the London Mathematical Society},
      volume={58},
      number={1},
       pages={e70186},
         doi={10.1112/blms.70186},
}

\bib{miyashita2024canonical}{article}{
  author    = {Miyashita, Sora}
  author = {Varbaro, Matteo},
  title     = {The canonical trace of Stanley--Reisner rings that are Gorenstein on the punctured spectrum},
  journal   = {International Mathematics Research Notices},
  volume   = {2025},
  number = {12},
  year   = {June 2025},
}


\bib{miyazaki2021Gorenstein}{article}{
      author={Miyazaki, Mitsuhiro},
       title={On the Gorenstein property of the Ehrhart ring of the stable set polytope of an h-perfect graph},
        date={2021},
     journal={International Electronic Journal of Algebra},
      volume={30},
      number={30},
       pages={269\ndash 284},
}

\bib{miyazaki2024non}{article}{
      author={Miyazaki, Mitsuhiro},
      author={Page, Janet},
       title={Non-Gorenstein loci of Ehrhart rings of chain and order polytopes},
        date={2024},
     journal={Journal of Algebra},
      volume={643},
       pages={241\ndash 283},
}

\bib{moscariello2025nearly}{article}{
      author={Moscariello, Alessio},
      author={Strazzanti, Francesco},
       title={Nearly Gorenstein numerical semigroups with five generators have bounded type},
        date={2025},
     journal={Communications in Algebra},
      volume={53},
      number={10},
       pages={4041\ndash 4052},
         doi={10.1080/00927872.2025.2455457},
}

\bib{ogoma1984existence}{article}{
      author={Ogoma, Tetsushi},
       title={Existence of dualizing complexes},
        date={1984},
     journal={Journal of Mathematics of Kyoto University},
      volume={24},
      number={1},
       pages={27\ndash 48},
}

\bib{stanley2007combinatorics}{book}{
      author={Stanley, Richard~P},
       title={Combinatorics and commutative algebra},
   publisher={Springer Science \& Business Media},
        date={2007},
      volume={41},
}

\bib{varbaro2024lefschetz}{article}{
      author={Varbaro, Matteo},
      author={Yu, Hongmiao},
       title={Lefschetz duality for local cohomology},
        date={2024},
     journal={Journal of Algebra},
      volume={639},
       pages={498\ndash 515},
}

\bib{villarreal2001monomial}{book}{
      author={Villarreal, Rafael~H},
       title={Monomial algebras},
   publisher={Marcel Dekker New York},
        date={2001},
      volume={238},
}

\end{biblist}
\end{bibdiv}

\end{document}